\documentclass[11pt,reqno]{amsart}
\usepackage{amsmath,amsfonts,amsthm,amscd,amssymb,graphicx}
\usepackage{epstopdf}
\numberwithin{equation}{section}
\usepackage{fullpage}
\usepackage{color}

\numberwithin{equation}{section}

\newtheorem{theorem}{Theorem}[section]

\newtheorem{lemma}{Lemma}[section]

\newtheorem{proposition}{Proposition}[section]

\theoremstyle{definition}
\newtheorem{definition}{Definition}[section]
\newtheorem{remark}{Remark}[section]

\newcommand{\RR}{{\mathbb R}}
\newcommand{\ds}{\displaystyle}

\newcommand\bp{\begin{pmatrix}}
\newcommand\ep{\end{pmatrix}}
\newcommand\be{\begin{equation}}
\newcommand\ee{\end{equation}}

\renewcommand\o{\omega}

\title{Traveling waves for  bistable \ reaction-diffusion-convection equations  \\ with discontinuous density-dependent coefficients}

\author{Pavel Dr\'{a}bek}
\address{Department of Mathematics and NTIS, Faculty of Applied Sciences, University of West Bohemia, Univerzitní 8, Plze\'{n} 30100, Czech Republic}
\email{pdrabek@kma.zcu.cz}

\author{Soyeun Jung}
\address{Division of International Studies, Kongju National University, Gongju-si, Chungcheongnamdo, South Korea}
\email{soyjung@kongju.ac.kr}

\author{Eunkyung Ko}
\address{Department of Mathematics, College of Natural Sciences, Keimyung University, Daegu  42601, South Korea}
\email{ekko@kmu.ac.kr}

\author{Michaela Zahradn\'{i}kov\'{a}}
\address{Department of Mathematics, Faculty of Science, University of South Bohemia,  Brani\v{s}ovsk\'{a}~1760, 370~05~\v{C}esk\'{e}~Bud\v{e}jovice, Czech Republic}
\email{zahram05@prf.jcu.cz}

\begin{document}

\begin{abstract} 
Continuing our previous study \cite{DJKZ} on the monostable reaction–diffusion–convection equation, we analyze the bistable case under weak regularity assumptions. Our approach applies monostable results on the subintervals where the reaction term $g$ has constant sign, thereby establishing both existence and nonexistence of bistable traveling wave solutions. We extend the results of \cite{MMM04}, obtained for $p=2$ under higher regularity assumptions ($d \in C^1[0,1]$, $g,h \in C[0,1]$), to the $p$-Laplacian with $p>1$ in our weak regularity setting. 
\end{abstract}


\maketitle


\smallskip


\smallskip


\smallskip

\section{Introduction}

We study traveling wave solutions of the nonlinear reaction–diffusion–convection equation
\begin{equation}
	\label{rdc-eq}
	u_t+h(u)u_x=\left[d(u)|u_x|^{p-2}u_x\right]_x+g(u), \quad (x,t) \in \mathbb{R} \times [0,+\infty).
\end{equation}
Here $p>1$, $h(u)$ is a nonlinear convective velocity, $d(u)$ is a diffusion coefficient, and $g(u)$ is a bistable reaction term, i.e., there exists $s_* \in (0,1)$ such that 
\be \label{r}
g(0)=g(1)=g(s_*)=0, \quad \text{$g<0$ in $(0,s_*)$}, \quad  \text{$g>0$ in $(s_*, 1)$}. 
\ee

Reaction–diffusion–convection equations naturally arise in a wide range of biological, chemical, and physical models, including population dynamics, combustion processes, and chemical reactions. They capture the interaction between local reaction kinetics, diffusive spreading, and advective transport. A central theme in this context is the study of traveling wave solutions, particularly their existence, as they describe the propagation of profiles connecting distinct equilibrium states. A traveling wave solution is a solution of the form $u(x,t)=U(x-ct)$, where $U$ is the profile, and $c$ denotes the  speed of wave propagation (wave speed for short).

In this paper, we investigate monotone decreasing traveling wave solutions, connecting the equilibria $1$ and $0$, under weak regularity assumptions on the coefficients. We assume only that $g$ is continuous on $[0,1]$, while $d$ and $h$ may have finitely many jump discontinuities. Moreover, $d$ may degenerate or have a singularity at one or both endpoints. The precise assumptions on $d$ and $h$ will be given in the next section.

The existence and qualitative properties of traveling waves strongly depend on the type of the reaction term. Three classical classes are typically distinguished: Fisher–KPP (monostable), combustion, and Nagumo (bistable) nonlinearities; see \cite[Fig.1]{DJKZ} for their typical sign conditions. It is well known that the existence of traveling waves differs fundamentally between the monostable and the bistable or combustion cases. In the monostable case, traveling waves exist for all speeds greater than a threshold wave speed (\cite{AW, DrTa, Enguica, MM05}). By contrast, in the bistable or combustion case the wave speed is uniquely determined, so that a traveling wave profile exists only for a single critical value $c^*$ (\cite{AW, MMM04, Strier}).

In the absence of convection (i,e., $h \equiv 0$), the theory of traveling waves has been extensively studied under various regularity assumptions (e.g., see \cite{DJKZ, DZ, DZ25, GM} and the references therein). In weak regularity settings such as ours, Drábek and Zahradníková analyzed the bistable case in \cite{DZ, DZ21} and the monostable case in \cite{DZ22}, establishing existence results and asymptotic properties of traveling wave profiles.

When convection is present, however, the dynamics become considerably more intricate. The convective term may shift the threshold wave speed in the monostable case, or even destroy the existence of traveling waves in the bistable or combustion case. For example, the bistable problem \eqref{rdc-eq} with $p=2$ has been investigated in \cite{MMM04} under the assumptions $d \in C^1[0,1]$ and $g,h \in C[0,1]$. In that work, the authors showed that the additional convection term $h$ can cause traveling wave profiles to disappear, despite their existence in the pure reaction–diffusion case. Their results demonstrated that existence depends not only on diffusion and reaction, but also on the local behavior of $g$ and $h$ near the unstable equilibrium $s_*$, together with appropriate sign conditions on $h$.

In our previous work \cite{DJKZ}, we studied the monostable case with convection under the same weak regularity setting as in the present paper. There we introduced the concept of non-smooth traveling wave profiles to handle discontinuous diffusion exhibiting degeneracy or singularities. This extended the earlier smooth results \cite{MM02,MM05}. In particular, it showed that the convection term may shift the threshold wave speed $c^*$ even under such weak regularity.

The present paper focuses on the bistable case with convection under weak regularity assumptions. Our main contribution is to establish existence and nonexistence results of monotone decreasing traveling waves in this setting, thereby showing that the results of \cite{MMM04} extend to the $p$-Laplacian with $p>1$ and under only weak regularity of the coefficients. In this sense, the paper bridges the smooth bistable analysis of \cite{MMM04} and the weakly regular monostable framework of \cite{DJKZ}.

From a methodological perspective, our strategy relies on the reduction to a first-order boundary value problem developed in \cite{DJKZ}. Unlike the monostable case, however, monotonicity of traveling wave profiles is not guaranteed in the bistable setting. Since this reduction requires monotonicity, we first introduce a suitable class of solutions for which monotonicity is ensured (see Remark \ref{Mrem1} and \ref{Mrem2} for details). We then apply the monostable techniques separately to the first-order problems on $[0,s_*]$ and $[s_*,1]$, where $g$ changes sign. This procedure leads to two threshold speeds, $c_F$ and $c_B$, which always satisfy $c_F \leq c_B$. A global traveling wave profile connecting $1$ to $0$ can exist only if $c_F < c_B$. In that case, the unique admissible wave speed $c^*$ lies strictly between them. The central task of this paper is therefore to prove that the inequality $c_F < c_B$ is sufficient for the existence of a unique wave speed $c^*$, and to provide sufficient conditions on $d$, $h$, and $g$ under which this inequality holds. We remark that this overall strategy follows closely the approach employed in \cite{MMM04} under stronger regularity assumptions, which we now adapt to the present weak framework.

The paper is organized as follows. In Section \ref{MainResults}, we state the main results, together with the precise assumptions on $g$, $d$, and $h$, and recall the notion of traveling wave profiles introduced in \cite{DJKZ}. In Section \ref{Equivalent first order boundary value problem}, we reformulate the second-order traveling wave equation as an equivalent first-order boundary value problem, which provides the analytical framework for the rest of the paper. Section \ref{Forward and Backward Initial Value Problems} studies the  first-order initial and terminal value problems on $[0,1]$, establishing positivity, uniqueness, continuous dependence, and monotonicity with respect to the wave speed. Section \ref{BVP on subintervals} focuses on auxiliary boundary value problems posed on the subintervals $[0,s_*]$ and $[s_*,1]$, from which the threshold speeds $c_F$ and $c_B$ are derived. Finally, in Section \ref{results of the first order ode}, we combine these ingredients to prove the main theorems on existence and nonexistence of bistable traveling waves.

An interesting direction for future research is to consider combustion case under the same weak regularity assumptions as in this paper. Recent work by Drábek and Zahradníková has addressed this problem in \cite{DZ25}, but only in the restricted regime where $c \geq \sup_{t \in (0,1)} h(t)$. In particular, their results do not cover the full necessary range $c \geq h(0)$ for the existence of traveling waves. We believe that the techniques developed here for monostable and bistable nonlinearities, which apply regardless of the sign of $c-h(U)$, can be adapted to treat the combustion case as well. Another natural question concerns the existence of traveling wave solutions under weaker assumptions on the reaction term. In particular, much less is known when the reaction term is discontinuous. In the monostable setting, such existence results were recently obtained in \cite{GM}. Extending these results to the bistable or combustion case remains an interesting problem.

\smallskip

\section{Main results} \label{MainResults}

We begin by stating the assumptions on the reaction term $g$, the diffusion coefficient $d$, and the convection velocity $h$.

\medskip

\noindent \textbf{(H1)} $g:[0,1] \rightarrow \RR$ is continuous (not necessarily smooth) and satisfies
\be \label{r}
g(0)=g(1)=g(s_*)=0, \quad \text{$g<0$ in $(0,s_*)$}, \quad  \text{$g>0$ in $(s_*, 1)$}. 
\ee
\noindent \textbf{(H2)} $d:[0,1] \rightarrow \RR$ is nonnegative, lower semicontinuous   on $[0, 1]$ and strictly positive in $(0,1)$. 
There exist points $0=m_0<m_1< \cdots < m_r  <  m_{r+1} =1 $ such that $d$ has jump discontinuities at $m_i$ ($i=1, 2, \cdots, r$) and
\be
d|_{(m_i, m_{i+1})} \in C(m_i, m_{i+1}), ~  i=0, 1, \cdots, r.
\notag
\ee
In addition,  for any $g$ satisfying \textbf{(H1)}, $d^{\frac{1}{p-1}}g  \in L^1(0,1)$ and  
\be \label{p}
\int_0^1 (d(t))^{\frac{1}{p-1}}g(t) dt \geq 0. 
\ee

\noindent \textbf{(H3)} $h:[0,1] \rightarrow \RR$ is bounded and continuous at $s_*$. There exist points $$0=k_0<k_1<k_2< \cdots <k_q <k_{q+1}=1$$ such that $h$ has jump discontinuities at $k_j$ ($j=1, 2, \cdots, q$), and
\be
h|_{(k_j, k_{j+1})} \in C(k_j, k_{j+1}), \quad j=0,1, \cdots, q.
\notag
\ee
Since $h$ is bounded on $[0,1]$, we denote
\be \label{maxminh}
h_M:=\ds \sup_{s_* \leq t \leq 1} h(t), \quad h_m:=\ds \inf_{0 \leq t \leq s_*} h(t).
\ee

\medskip

When seeking traveling wave solutions of \eqref{rdc-eq}, we formally set 
\be
u(x,t)=U(z),  \quad z=x-ct, 
\notag
\ee
where $c$ denotes the wave speed. Substituting this ansatz into \eqref{rdc-eq} yields the following ordinary differential equation (ODE):
\begin{equation}\label{eqU}
 (d(U)|U'|^{p-2} U')' +(c-h(U) )U' +g(U) =0,
 \end{equation}
where primes denote differentiation with respect to the traveling wave coordinate $z$.

The primary objective of this project is to study a nonincreasing solution of \eqref{eqU} satisfying the boundary conditions
\be \label{bc}
\lim_{z \rightarrow -\infty} U(z) =1 ~~\mbox{and}~~\lim_{z\rightarrow \infty} U(z) = 0.
\ee
More precisely, we consider a nonincreasing profile $U$ that is strictly decreasing at every point where $U(z)\in(0,1)$; that is, in our setting, $U'(z)<0$ whenever the derivative exists, and both one-sided derivatives are negative at any point where $U'$ does not exist. We thereore define the class of solutions to be studied as follows.

\subsection{Definition of solution}\label{DefSol}

Since $d(U)$ and $h(U)$ may be discontinuous on $[0,1]$, we cannot expect a classical solution $U=U(z)$ of \eqref{eqU}, and hence we need to define more general solutions of \eqref{eqU}. Moreover, the bistable reaction term does not guarantee monotonicity of the profile $U$. Therefore, we define solutions only within the class of monotone functions.

Let $U:\mathbb{R}\to[0,1]$ be a monotone continuous function. We denote
\be \label{MKN}
\begin{split}
M_U: & = \{z \in \mathbb{R} : U(z) = m_i, ~ i=1,2, \dots, r \}, \\
K_U: & = \{z \in \mathbb{R} : U(z) = k_j, ~j=1,2, \dots q \}, \\
N_U: & = \{z \in \mathbb{R}: U(z)=0 ~\mbox{or}~ U(z)=1  \}, \\
\text{and} \quad D_U: & = \partial M_U \cup \partial N_U. 
\end{split}
\ee
Then $M_U$, $K_U$, and $N_U$ are closed sets. In particular, we set 
\be\label{Nend}
N_U  = (-\infty, z_0] \cup [z_1, \infty),
\ee
where $-\infty \leq z_0 <z_1 \leq \infty$. Following \cite{DJKZ, DZ22}, we now define the notion of monotone solution used in this paper. 

\begin{definition} \label{sdef}
A  monotone  continuous function $U :\mathbb{R} \rightarrow [0,1] \in C(\RR)$ is called a monotone solution of \eqref{eqU} if it satisfies the following properties:

\smallskip
\begin{itemize}

\item[$(a)$]  $U \in C^1(\mathbb{R} \setminus D_U),$ and $U'(z) \neq 0$ for all $z \in \mathbb{R} \setminus D_U$ whenever $U(z) \in (0,1)$. 

\smallskip

\item [$(b)$] For any $z \in \partial M_U$ there exist finite nonzero one-sided derivatives $U'(z-), U'(z+)$ and
\be \label{Ldef}
L(z) := |U'(z-)|^{p-2} U'(z-) \lim_{\xi \rightarrow z-} d(U(\xi)) = |U'(z+)|^{p-2} U'(z+)\lim_{\xi \rightarrow z+ } d(U(\xi)).
\ee

\item [$(c)$] A function $v: \mathbb{R} \rightarrow \mathbb{R}$ defined by
 \be\label{v}
 v(z) :=
 \left\{\begin{array}{lll}
 d(U)|U'(z)|^{p-2} U'(z), & z \not\in \partial M_U \cup  \partial N_U,\\
 0, & z\in  \partial  N_U,\\
 L(z), &z \in \partial M_U
 \end{array}\right.
 \ee
is continuous and  for any $z , \hat{z} \in \mathbb{R}$
\begin{equation}\label{eqv}
 v(\hat{z}) - v(z) +c (U(\hat{z}) - U(z)) - (H (U(\hat{z})) - H(U(z))) + \int_z^{\hat{z}} g(U(\tau)) d\tau =0,
\end{equation}
where $H(U):= \ds\int_0^U h(s) ds$. Moreover, $\ds \lim_{z \rightarrow \pm \infty} v(z)=0$.

\end{itemize}
 \end{definition}



\begin{remark} \label{Mrem1} 
By applying \cite[Lemma 3.2]{DJKZ} to each interval $(0,s_*)$ and $(s_*,1)$, on which the reaction term $g$ has a constant sign, both $U'(z-)$ and $U'(z+)$ are nonzero whenever $U(z) \in (0,1)\setminus \{s_*\}$. However, since $g(s_*)=0$, the nonvanishing of the one-sided derivatives $U'(z-)$ and $U'(z+)$ is not guaranteed when $U(z)=s_*$. For this reason, we explicitly require the nonvanishing of the derivative and of the one-sided derivatives in Definition \ref{sdef}(a) and (b), respectively. However, for $1<p\le2$, such nonvanishing, even when $U(z)=s_*$, can in fact be proved by adapting the argument of \cite[Theorem~5]{MMM04} (which was established for $p=2$) to our setting. The proof is provided in the appendix for the interested reader.  In the appendix, the nonvanishing for $p>2$ can also be proved, but only under an additional assumption on $d^{\frac{1}{p-1}}g$. As we aim to assume nothing beyond continuity of the reaction term $g$, the case $p>2$ without such additional assumptions remains open.  
\end{remark}

\begin{remark} \label{Mrem2} 
Since $h$ is continuous at $s_*$, we have $s_* \notin K_U$. This implies $\operatorname{int} K_U = \emptyset$, that is, $K_U = \partial K_U$ (see \cite[Remark 2.1]{DJKZ}). Moreover, the nonvanishing of both $U'(z^+)$ and $U'(z^-)$ also implies that $M_U = \partial M_U$. Indeed, without this nonvanishing property, $M_U$ may contain at most one interval $(a, b)\subset \operatorname{int} M_U$, with $-\infty \le a < b \le \infty$, on which $U(z) \equiv s_*$ and $U'(z) \equiv 0$. Consequently, under our definition of monotone solution, each of $M_U$ and $K_U$ consists only of finitely many points. 
\end{remark}

\begin{remark} \label{property of U}
We now mention several basic properties of profiles $U$ satisfying Definition \ref{sdef}, which were already obtained in \cite{DJKZ} and are included here for completeness.

\smallskip

\noindent (i) For any $z \in \partial M_U$, $U'(z-) \not= U'(z+)$. Otherwise, one would have
 $$ |U'(z-)|^{p-2} U'(z-) \lim_{\xi \rightarrow z-} d(U(\xi)) \not= |U'(z+)|^{p-2} U'(z+)\lim_{\xi \rightarrow z+ } d(U(\xi)),$$ 
since $ d(U)$ has a jump discontinuity at $m_i$, which contradicts Definition \ref{sdef}(b). 
 In contrast, for $z \in   K_U\setminus \partial M_U$,
 the derivative $U'(z)$ exists. 
  Indeed, if $U'(z-) \not= U'(z+)$ for some $z \in   K_U\setminus \partial M_U,$
then
$$d(U(z-))|U'(z-)|^{p-2} U'(z-) \not= d(U(z+))|U'(z+)|^{p-2} U'(z+) $$
because $d(U)$ is continuous at such points,  which  contradicts the continuity of $v(z)$.




\smallskip


\noindent $(ii)$ Let $z \not\in \partial M_U \cup  K_U \cup \partial N_U$.  
Assume that $|\tau|$ sufficiently  small so that $ \hat{z} = z+ \tau  \not\in \partial M_U \cup  K_U \cup \partial N_U.$
Dividing \eqref{eqv} by $\tau$ and taking $\tau \rightarrow 0$ we obtain
$$ \lim_{\tau \rightarrow 0} \frac{v(z+\tau ) - v(z)}{\tau } +(c-h(U(z)))U'(z)  +g(U(z)) =0$$
since $U'(z)$ and $\frac{d}{dz}H(U(z))$ exist. 
Hence, $v'(z)$ exists  for all  $z \not\in \partial M_U \cup  K_U \cup \partial N_U$,  and it follows that
\be
v'(z)+(c -h(U(z)))U'(z) +g(U(z))=0,
\notag
\ee
which implies that $v \in C^1 (\mathbb{R} \setminus (\partial M_U \cup  K_U \cup \partial N_U) ).$

\smallskip

\noindent $(iii)$ Let $z \in \partial M_U \cup K_U$. As in the previous case,  
 letting  $\tau \rightarrow 0-$ in \eqref{eqv} yields
\be
v'(z-)+\Big(c -\lim_{\xi \rightarrow z-}h(U(\xi))\Big)U'(z-) +g(U(z))=0,
\notag
\ee
which ensures the existence of $v'(z-)$. Likewise,   taking the limit as  $\tau \rightarrow 0+,$ the right derivative $v'(z+)$ exists and satisfies
\be
v'(z+) +\Big(c -\lim_{\xi \rightarrow z+}h(U(\xi))\Big)U'(z+) +g(U(z))=0.
\notag
\ee
In particular, if $z \in K_U \setminus  \partial M_U$, then $U'(z-)=U'(z+)$ and
\be
v'(z\pm)+\Big(c -\lim_{\xi \rightarrow z\pm}h(U(\xi))\Big)U'(z) +g(U(z))=0.
\notag
\ee
If $z \in \partial M_U \setminus K_U$, then $\ds \lim_{\xi \rightarrow z-}h(U(\xi))=\lim_{\xi \rightarrow z+}h(U(\xi))=h(U(z))$ and
\begin{equation}\label{bMU}
v'(z\pm)+\Big(c -h(U(z))\Big)U'(z\pm) +g(U(z))=0.
\end{equation}
\end{remark}

\subsection{Existence and nonexistence of traveling waves}

We now present our main results on the existence and nonexistence of monotone decreasing solution to the second-order boundary value problem (BVP)
\begin{equation}\label{odeU}
\begin{cases}
(d(U)|U'|^{p-2}U')' + (c - h(U))U' + g(U) = 0, \quad z \in \mathbb{R}, \\
\displaystyle\lim_{z \to -\infty} U(z) = 1, \quad \text{and} \quad \displaystyle\lim_{z \to +\infty} U(z) = 0.
\end{cases}
\end{equation}

We state three main theorems: two concerning the existence of traveling wave solutions and one concerning their nonexistence. These correspond to the main results obtained in \cite[Theorems 1–3]{MMM04}, but we show that the same conclusions remain valid under the weak regularity assumptions \textbf{(H1)}–\textbf{(H3)} for $p>1$.

The first theorem provides sufficient conditions (see \eqref{MExcondition_1} and \eqref{MExcondition_2}) for the existence of nonincreasing traveling wave solutions to \eqref{odeU}. These conditions, originally introduced in \cite[Theorem 1]{MMM04}, describe the proper balance among diffusion, convection, and reaction near the unstable equilibrium $s_*$. In particular, they clarify the role of the convection term $h$ compared to the nondegeneracy condition $g'(s_*) \neq 0$ assumed in \cite{GK}. It also specifies the admissible range of the unique wave speed $c^*$ (see \eqref{Mrange}) and provides qualitative properties of the corresponding profile $U$. 

Before stating the theorem, we recall \eqref{maxminh} and introduce the following quantities:
\be \label{MFmu}
\mu_1 := \sup_{t \in (0, s_*)} \frac{-(d(t))^{\frac{1}{p-1}}g(t)}{(s_*-t)^{p'-1}}>0, \quad \mu_2: = \sup_{t\in (s_*,1)} \frac{(d(t))^{\frac{1}{p-1}}g(t)}{(t-s_*)^{p'-1}}>0.
\ee
These quantities will be used to describe the admissible range of the wave speed $c^*$.

\begin{theorem} \label{existencetws} Assume that $g$, $d$ and $h$ satisfy \textbf{(H1)}–\textbf{(H3)}. If either condition 
\be\label{MExcondition_1}
\liminf_{t \rightarrow s_*^-}\frac{(s_*-t)^{p'-1}(h(s_*)-h(t))}{-(d(t))^{\frac{1}{p-1}}g(t)}>-\infty
\ee
or
\be\label{MExcondition_2}
\limsup_{t \rightarrow s_*^+}\frac{(t-s_*)^{p'-1}(h(s_*)-h(t))}{(d(t))^{\frac{1}{p-1}}g(t)}<\infty
\ee
holds, then the problem \eqref{odeU} admits a unique profile  $U=U(z)$, $z\in \RR$,  satisfying the following properties \textup{(i)}–\textup{(iv)} if and only if $c=c^*$, where $c^*$ satisfies 
\be \label{Mrange}
h_m-(p')^{\frac{1}{p'}} p^{\frac{1}{p}} \mu_1^{\frac{1}{p'}} < c^*  < h_M+(p')^{\frac{1}{p'}} p^{\frac{1}{p}} \mu_2^{\frac{1}{p'}}. 
\ee  
The properties \textup{(i)}–\textup{(iv)} are as follows: 
\begin{itemize}
\item[(i)] $U$ is strictly decreasing on the open interval $(z_0, z_1)=\{z \in \RR: 0<U(z)<1\}$ where $-\infty \leq z_0 < 0 <z_1 \leq +\infty$, $U(0)=\frac{1}{2}$, $U(z)=1$ for $z \in (-\infty, z_0]$, and $U(z)=0$ for $z \in [z_1, +\infty)$. In particular, for every $z\in(z_0,z_1)$, both one-sided derivatives $U'(z-)$ and $U'(z+)$ exist and are nonzero.

\smallskip
	
\item[(ii)] $M_U=\{\zeta_1, \zeta_2, \cdots, \zeta_r \}$ and $U \in C(\RR)$ is a piecewise $C^1(\RR)$--function in the sense that
\be
U|_{(\zeta_i, \zeta_{i+1})} \in C^1(\zeta_i, \zeta_{i+1}), \quad i=0, 1, \cdots, r, \quad \zeta_0 = z_0, \quad \zeta_{r+1} = z_1,
\notag
\ee
and one-sided derivatives $\ds U'(\zeta_i \pm)=\lim_{z \rightarrow \zeta_i \pm}U'(z)$, $i=1, 2, \cdots, r$, exist and are finite.
	
\smallskip

\item[(iii)] $U$ satisfies $\ds \lim_{z \rightarrow z_0+} d(U(z))|U'(z)|^{p-2}U'(z)=\lim_{z \rightarrow z_1-} d(U(z))|U'(z)|^{p-2}U'(z)=0$
and the transition condition
\be
|U'(\zeta_i-)|^{p-2}U'(\zeta_i-) \lim_{z \to \zeta_i-} d(U(z))=|U'(\zeta_i+)|^{p-2}U'(\zeta_i+) \lim_{z \to \zeta_i+} d(U(z)),  \quad  i=1, 2, \cdots, r.
\notag
\ee

\smallskip

\item[(iv)] $U$ satisfies the problem \eqref{odeU} (in a classical sense) on $\RR \setminus \{\partial M_U \cup K_U \cup \partial N_U\}$.
\end{itemize}

\medskip

\end{theorem}

When $h$ is constant, both conditions \eqref{MExcondition_1} and \eqref{MExcondition_2} are automatically satisfied, and Theorem \ref{existencetws} holds trivially. Moreover, if either condition
\be
\liminf_{t \rightarrow s_*^-} \frac{g(t)}{(s_*-t)^{p'-1}}>0 \quad \text{or} \quad \liminf_{t \rightarrow s_*^+} \frac{g(t)}{(t-s_*)^{p'-1}}>0
\notag
\ee
holds, then either \eqref{MExcondition_1} or \eqref{MExcondition_2} is also fulfilled, thereby recovering the classical existence result (\cite{GK}) without imposing any restriction on $h$. 

\medskip

The next theorem provides another sufficient condition for the existence of nonincreasing traveling wave solutions to \eqref{odeU}.
When $h$ is nondecreasing near $s_*$, the conditions \eqref{MExcondition_1} and \eqref{MExcondition_2} are automatically satisfied, regardless of the behavior of $d^{\frac{1}{p-1}}g$ near $s_*$.
Therefore, the following condition \eqref{Mexistencecondition} can be regarded as a weaker form of the monotonicity assumption on $h$.

\begin{theorem} \label{existencetws2} Assume that $g$, $d$ and $h$ satisfy \textbf{(H1)}–\textbf{(H3)}. If $h$ further satisfies 
\be \label{Mexistencecondition}
\int_{s_*}^{t_*} (h(s_*)-h(t))dt \leq 0 \quad  \text{for some $t_* \in [0,1]\setminus\{s_*\}$}, 
\ee
 then the problem \eqref{odeU} admits a unique profile $U=U(z)$, $z\in \RR$,  satisfying the above properties \textup{(i)}–\textup{(iv)} if and only if $c=c^*$, where $c^*$ satisfies \eqref{Mrange}. 
\end{theorem}

The final result establishes a sufficient condition for the nonexistence of nonincreasing traveling wave solutions to \eqref{odeU}. In contrast to the previous existence theorems, this result characterizes situations where no traveling wave connecting $1$ and $0$ can exist for any wave speed $c \in \mathbb{R}$.
Indeed, when a certain value of $c$ allows the construction of two separate nonincreasing profiles connecting $1$ to $s_*$ and $s_*$ to $0$, there exists no global profile connecting $1$ and $0$ (see Proposition \ref{existenceprop}). The following conditions \eqref{Mnonh}-\eqref{MnonFf} provide a sufficient criterion ensuring that such a critical value corresponds precisely to $c = h(s_*)$, which in turn precludes the existence of a nonincreasing traveling wave connecting the two stable equilibria.

\begin{theorem}\label{nonexistencetws} Assume that $g$, $d$ and $h$ satisfy \textbf{(H1)}–\textbf{(H3)}. Further suppose that 
\be \label{Mnonh}
\ds \int_{s_*}^t (h(s_*)-h(\tau))d\tau>0 \quad  \text{for all $t \in [0,1]\setminus\{s_*\}$}
\ee
holds together with 
\be \label{MnonBf}
(d(t))^{\frac{1}{p-1}}g(t)<\frac{1}{p'p^{p'/p}}(h(s_*)-h(t)) \Big[\int_{s_*}^t (h(s_*)-h(\tau))d\tau\Big]^{\frac{p'}{p}}, \quad t \in (s_*, 1),
\ee
and 
\be \label{MnonFf}
(d(t))^{\frac{1}{p-1}}g(t)>\frac{1}{p'p^{p'/p}}(h(s_*)-h(t)) \Big[ \int_{s_*}^t (h(s_*)-h(\tau))d\tau \Big]^{\frac{p'}{p}}, \quad t \in (0,s_*).
\ee
Then the problem \eqref{odeU} has no monotone decreasing solution for any $c \in \RR$. 
\end{theorem}

\begin{remark} Our nonexistence result differs slightly from that of \cite[Theorem 4]{MMM04}, which treats only the case $p=2$ and shows that the problem does not support any traveling wave solution connecting $0$ and $1$ for any $c \in \RR$. This stronger conclusion relies on the fact that, if the associated traveling wave ODE has a solution for a given speed $c$, then it also admits a monotone solution for the same $c$ (see \cite[Theorem 5]{MMM04}). Hence, the nonexistence of monotone solutions directly implies the nonexistence of all traveling waves connecting $0$ and $1$. However, in our framework, which covers the general case $1<p<\infty$, such a reduction to monotone solutions is not available. For this reason, our result is restricted to the nonexistence of monotone decreasing solutions.
\end{remark}

\section{Equivalent first-order BVP} \label{Equivalent first order boundary value problem}

In this section, we reformulate problem \eqref{odeU} as an equivalent first-order BVP. The next proposition establishes the correspondence between monotone decreasing solutions of \eqref{odeU} and positive solutions of the associated first-order problem. We recall that monotone decreasing solutions of \eqref{odeU} are already assumed to have negative one-sided derivatives at every point where $U(z) \in (0, 1)$, as in the monostable case. Since  the derivation of this equivalence was already given in our previous work on the monostable case (\cite[Section~3.2]{DJKZ}), we omit the proof here.

\begin{proposition} \label{prop} A piecewise $C^1$ function $U:\RR \rightarrow [0,1]$ is a monotone decreasing solution of \eqref{odeU}  if and only if $y:[0,1] \rightarrow \RR$ is a continuous positive solution of the first-order BVP
\begin{equation}\label{odey}
\begin{cases}
y'(t) = p' \left[ (c-h(t)) (y(t))^{\frac{1}{p}} -(d(t))^{\frac{1}{p-1}}g(t)  \right],  \quad \text{a.e. $t \in (0,1)$}, \\
 y(0)=0=y(1),
\end{cases}
\end{equation}
where $p>1$  with $\frac{1}{p}+\frac{1}{p'}=1$. In particular, a solution $U$ of \eqref{odeU} is uniquely determined (up to translation) by a solution $y$ of \eqref{odey}, and vice versa.
\end{proposition}

Thanks to Proposition~\ref{prop}, our analysis will henceforth be carried out in the first-order formulation, which provides the main framework for studying the existence and nonexistence of nonincreasing traveling wave profiles. Accordingly, in what follows we focus on positive solutions of the following first-order BVP
\begin{equation}\label{ode}
\begin{cases}
y'(t) = p' \left[ (c-h(t)) (y^+ (t))^{\frac{1}{p}} -f(t)  \right],  \quad \text{a.e. $t \in (0,1)$}, \\
 y(0)=0=y(1),
\end{cases}
\end{equation}
where $y^+(t) := \max\{y(t), 0\}$, $p>1$ with $\frac{1}{p}+\frac{1}{p'}=1$, $h \in L^{\infty}(0,1)$, and $f \in L^1(0,1)$ satisfying 
\be \label{fcondition}
f(0)=f(s_*)=f(1)=0, \quad \text{$f < 0$ on $(0, s_*)$, \quad $f > 0$ on $(s_*, 1)$}
\ee
for some $s_* \in (0,1)$. 

Since $U$ is not necessarily a classical solution of \eqref{odeU}, we adopt the notion of Carathéodory solutions for \eqref{ode} (see \cite[Chapter~3]{W}). That is, a function $y:[0,1]\to\RR$ is called a solution of \eqref{ode} if it is absolutely continuous on $[0,1]$, satisfies the differential equation almost everywhere in $(0,1)$, and fulfills the boundary conditions.

This first-order formulation will be used exclusively in the subsequent sections. At each stage of the analysis, appropriate assumptions on $f$ and $h$ will be imposed as needed to investigate the existence and nonexistence of positive solutions to $\eqref{ode}$. 

\medskip

Before proceeding, we derive a basic necessary condition on the wave speed $c$ for the  BVP \eqref{ode} to admit a positive solution under the integral sign assumption of $f$.

\begin{lemma}
Assume that 
\be \label{pf}
\int_0^1 f(t) dt \geq 0
\ee
holds and the BVP \eqref{ode} admits a positive solution. Then
\be \label{bnc1}
c \geq \inf_{0 \leq t \leq 1}{h(t)}.
\ee
In particular, if the inequality in \eqref{pf} is strict, then
\be\label{bnc2}
c > \inf_{0 \leq t \leq 1} h(t),
\ee
whereas if equality holds in \eqref{pf}, then
\be\label{bnc3}
\inf_{0 \leq t \leq 1} h(t) \leq c \leq \sup_{0 \leq t \leq 1} h(t).
\ee
\end{lemma}
\begin{proof} Let $y=y(t)$ be a positive solution of \eqref{ode}. Integrating \eqref{ode} from $0$ to $1$ yields
\be
0=p' \left[ \int_0^1 (c-h(t)) (y(t))^{\frac{1}{p}}dt -\int_0^1 f(t) dt \right],
\notag
\ee
which gives
\be
c=\frac{\int_0^1 h(t) (y(t))^{\frac{1}{p}}dt + \int_0^1 f(t) dt }{\int_0^1 (y(t))^{\frac{1}{p}}dt}. 
\notag
\ee
Thus it follows directly from \eqref{pf} that \eqref{bnc1}–\eqref{bnc3} hold.
\end{proof}

\section{Initial  and Terminal Value Problems} \label{Forward and Backward Initial Value Problems}

As a first step,  we study the initial  and terminal value problems (IVP and TVP) on $[0, 1]$, depending on a parameter $c \in \RR$:
\begin{equation}\label{y2}
\begin{cases}
y'(t) = p' \left[ (c-h(t)) (y^+ (t))^{\frac{1}{p}} -f(t)  \right],  \quad \text{a.e. $t \in (0, 1)$}, \\
 y(0)=0, 
\end{cases}
\end{equation}
and
\begin{equation}\label{y1}
\begin{cases}
y'(t) = p' \left[ (c-h(t)) (y^+ (t))^{\frac{1}{p}} -f(t)  \right],  \quad \text{a.e. $t \in (0, 1)$},\\
 y(1)=0. 
\end{cases}
\end{equation}
Here $y^+(t):=\max \{y(t),0\}$, $p>1$ with $\frac{1}{p}+\frac{1}{p'}=1$, $h \in L^{\infty}(0,1)$, and $f \in L^1(0,1)$ satisfying \eqref{fcondition}. Throughout, we further assume 
\be  \label{semicont}
f \ \text{is upper semicontinuous on } (0,s_*) \ \text{and lower semicontinuous on } (s_*,1).
\ee

To distinguish the two problems, for each $c\in \RR$ we denote by $\tilde y_c$ the solution of the IVP \eqref{y2} on $[0, 1]$, and by $\hat y_c$ the solution of the TVP \eqref{y1} on $[0, 1]$.

Many arguments here parallel those of the monostable case $f>0$ on $(0,1)$, where the problem was treated as a TVP on $[0, 1]$ (see \cite[Section~4]{DJKZ}). In the present bistable setting, we state lemmas on $(0,s_*)$ (where $f<0$) for the IVP and on $(s_*,1)$ (where $f>0$) for the TVP. Backward arguments for the TVP are identical to the monostable analysis restricted to $(s_*,1)$, while the  forward case  follows by symmetry. Thus, whenever the monostable proof applies without modification, we simply refer to \cite{DJKZ}. When a modification is needed, it is given for the forward case only, with the  backward case omitted.

Although the statements are formulated on subintervals $(0,s_*)$ and $(s_*,1)$, each property actually extends to the maximal subinterval where the corresponding solution remains positive. We highlight this after the relevant lemmas, since it is essential for solving the BVP \eqref{ode}. This motivates posing both IVP \eqref{y2} and TVP \eqref{y1} on $[0,1]$.

We begin with existence results for $\tilde y_c$ and $\hat y_c$ on $[0,1]$ in the sense of Carathéodory. For these proofs, neither the sign condition \eqref{fcondition} nor the semicontinuity assumption \eqref{semicont} is required; these will only be used later in the analysis.

\begin{lemma} \label{cara_existence} Fix $c \in \RR$. The  IVP \eqref{y2} $($resp. the TVP \eqref{y1}$)$ admits at least one global solution $\tilde y_c=\tilde y_c(t)$ $($resp. $\hat y_c=\hat y_c(t)$$)$ on $[0, 1]$ in the sense of Carath\'{e}odory.
\end{lemma}
\begin{proof} See \cite[Lemma 4.1]{DJKZ}. 
\end{proof}

We next state a positivity property on the subintervals $(0, s_*)$ and $(s_*, 1)$, which will play a key role in uniqueness and subsequent arguments. In our previous work on the monostable case with convection \cite{DJKZ}, establishing positivity without any sign condition on $c-h(t)$ was one of the main novelties.

\begin{lemma} \label{Fpositiveness}
Fix $c \in \RR$. Then $\tilde y_c(t) >0$ on $(0, s_*)$, and $\hat y_c(t) >0$ on $(s_*, 1)$. 
\end{lemma}
\begin{proof} See \cite[Lemma 4.2]{DJKZ}. 
\end{proof}

\smallskip

By the above lemma, the positive-part operator in $y^+$ can be omitted on $(0, s_*)$ and $(s_*, 1)$. This leads to the following uniqueness result. Compare with \cite[Lemma 4.3]{DJKZ}, the proof requires a modification to account for the maximal subinterval where the solution remains positive; for this we adapt the uniqueness argument of \cite{MMM04} to our setting.

\begin{lemma} \label{Uniqueness2}
Fix $c \in \RR$. Then $\tilde y_c=\tilde y_c(t)$ is the unique positive solution of \eqref{y2} on $[0, s_*]$,  and $\hat y_c=\hat y_c(t)$ is the unique positive solution  of \eqref{y1} on $[s_*, 1]$. 
\end{lemma}
\begin{proof} We prove the uniqueness for $\tilde y_c$ on $(0, s_*)$ only. Suppose that $\tilde y_1=\tilde y_1(t)$ and $\tilde y_2=\tilde y_2(t)$ are solutions of \eqref{y2} for a fixed $c\in \RR$. Since both are positive, set
\be
\tilde z_1(t):=(\tilde y_1(t))^{\frac{1}{p'}}>0 \quad \text{and} \quad \tilde z_2 (t):=(\tilde y_2(t))^{\frac{1}{p'}}>0, \quad t \in (0, s_*).
\notag
\ee
Recalling \eqref{y2} we directly compute
\be
\tilde z_1'(t) =c-h(t)-\frac{f(t)}{(\tilde z_1(t))^{1/(p-1)}} \quad \text{and} \quad
\tilde z_2'(t)  =c-h(t)-\frac{f(t)}{(\tilde z_2(t))^{1/(p-1)}}, \quad a.e. ~t \in (0, s_*).
\notag
\ee
Subtracting gives
\be \label{un1}
\begin{split}
(\tilde z_{1}-\tilde z_{2})^{\prime}(t)
& = \frac{f(t)}{(\tilde z_{1}(t))^{\frac{1}{p-1}} (\tilde z_{2}(t))^{\frac{1}{p-1}}} \Big[(\tilde z_{1}(t))^{\frac{1}{p-1}}- (\tilde z_{2}(t))^{\frac{1}{p-1}}\Big], \quad a.e. ~t \in (0, s_*).
\end{split}
\ee 
Applying the mean value theorem to $v \mapsto v^{\frac{1}{p-1}}$, for each $t \in (0, s_*)$ there exists $\theta(t)>0$ between $\tilde z_{1}(t)$ and $\tilde z_{2}(t)$ such that 
\be
\tilde z_{1}(t)-\tilde z_{2}(t)=\Big[(\tilde z_{1}(t))^{\frac{1}{p-1}}- (\tilde z_{2}(t))^{\frac{1}{p-1}}\Big](p-1)(\theta(t))^{1-\frac{1}{p-1}}.
\notag
\ee
Substituting this into \eqref{un1} and introducing a function
\be
\o(t):=\tilde z_{1}(t)- \tilde z_{2}(t) \quad \text{on} \quad [0, s_*], 
\notag
\ee
we deduce
\be
\o^{\prime}(t)=\frac{f(t)}{(p-1)(\tilde z_{1}(t))^{\frac{1}{p-1}} (\tilde z_{2}(t))^{\frac{1}{p-1}}(\theta(t))^{1-\frac{1}{p-1}}} \o(t), \quad a.e. ~t \in (0, s_*),
\notag
\ee
which is a first-order linear differential equation. We now define an integrating factor as
\be \label{IF}
F(t):=exp \Big[\int_{t}^{s_*} \frac{f(\tau)}{(p-1)(\tilde z_{1}(\tau))^{\frac{1}{p-1}} (\tilde z_{2}(\tau))^{\frac{1}{p-1}}(\theta(\tau))^{1-\frac{1}{p-1}}} d\tau \Big], \quad t \in (0, s_*]. 
\ee
Since $f<0$ on $(0, s_*]$ and the integrand lies in $L^1(t, s_*)$ for all $t \in (0, s_*)$, $0<F(t) \leq 1$ on $(0, s_*]$, and 
\be
(\o F)^{\prime}(t)=0, \quad a.e. ~t \in (0, s_*). 
\notag
\ee
Therefore,  $(\o F)(t)$ is constant on $[0, s_*]$. Since $(\o F)(0^+)=0$, it follows that $\o(t) \equiv 0$, and hence $\tilde z_{1}(t)=\tilde z_{2}(t)$ on $[0, s_*]$. 
\end{proof}

\begin{remark} \label{uniqueness of ps} The uniqueness result above extends  to the entire positivity subintervals of $\tilde y_c$ and $\hat y_c$. Fix $c\in\RR$. If $\tilde y_1$ and $\tilde y_2$ are positive solutions of \eqref{y2} on $(0,t_*)$ with $s_*<t_* <1$, then by Lemma~\ref{Uniqueness2} they coincide on $[0,s_*]$. Repeating the same argument on $[s_*,t_*]$ (with $f>0$) and introducing
\be
F(t):=exp \Big[\int_{t}^{s_*} \frac{f(\tau)}{(p-1)(\tilde z_{1}(\tau))^{\frac{1}{p-1}} (\tilde z_{2}(\tau))^{\frac{1}{p-1}}(\theta(\tau))^{1-\frac{1}{p-1}}} d\tau \Big], \quad t \in [s_*, t_*),   
\notag
\ee
one again obtains $\tilde y_1\equiv\tilde y_2$ on $[s_*,t_*]$. Hence $\tilde y_c$ is unique on its entire positivity interval; the same holds for $\hat y_c$.

Moreover, $\tilde y_c$ and $\hat y_c$ cannot intersect at a positive value unless they vanish simultaneously at $s_*$ or coincide identically. Suppose, to the contrary, that there exists $t_0\in(0,1)$ with $\tilde y_c(t_0)=\hat y_c(t_0)>0$ but $\tilde y_c\ne\hat y_c$ on a right neighborhood of $t_0$. Choose $t_1$ with $t_0 < t_1 <1$ such that $f$ has constant sign on $(t_0, t_1)$ and  both solutions remain positive on $[t_0, t_1]$. Setting $\o(t)=\tilde z_c(t)- \hat z_c(t)$ and introducing 
\be
F(t):=exp \Big[\int_{t}^{t_*} \frac{f(\tau)}{(p-1)(\tilde z_{c}(\tau))^{\frac{1}{p-1}} (\hat z_{c}(\tau))^{\frac{1}{p-1}}(\theta(\tau))^{1-\frac{1}{p-1}}} d\tau \Big], \quad t \in (t_0, t_1)
\notag
\ee 
with $t_*:=t_0$ if $f>0$ on $(t_0, t_1)$ and $t_*:=t_1$ if $f<0$ on $(t_0, t_1)$,  the same computation as before yields $\tilde z_c \equiv \hat z_c$ on $[t_0, t_1]$, a contradiction.  Likewise, if $\tilde y_c(t_1)=\hat y_c(t_1)>0$ but $\tilde y_c$ and $\hat y_c$ differ on a left neighborhood of $t_1$, choose $0<t_0<t_1$ so that $f$ has constant sign on $(t_0, t_1)$ and $\tilde y_c, \hat y_c>0$ on $[t_0, t_1]$. Then the preceding argument applies on $[t_0, t_1]$, yields a contradiction.  
\end{remark}

\smallskip

The uniqueness established above, together with the continuity in $c$ of $$\Gamma (t, y, c): =p'[(c-h(t))(y^{+}(t))^{1/p}-f(t)],$$ implies continuous dependence of $\tilde y_c$ and $\hat y_c$ on the parameter $c$. This follows from standard results on Carath\'{e}odory differential equations (see, e.g., Theorem $4.1$ and $4.2$ in Chapter $2$ of \cite{CL}). For completeness we record the statement below without proof.

\begin{lemma} \label{Cdop}
Fix $c_0 \in \RR$. As $c \rightarrow c_0$, the solutions $\tilde y_c=\tilde y_c(t)$ of \eqref{y2} $($resp. $\hat y_c=\hat y_c(t)$ of \eqref{y2}$)$ converge to $\tilde y_{c_0}$ $($resp. $\hat y_{c_0}$$)$ uniformly on every closed subinterval on which all these solutions remain positive. 
\end{lemma}

\smallskip

Next, we state the comparison results, which will be crucial for solving the BVPs  on $[0, s_*]$ and $[s_*, 1]$ in Section \ref{BVP on subintervals}. The following lemma is a simplified version of \cite[Lemma 4.5]{DJKZ}, obtained by merging alternatives $(a)$ and $(b)$. As the proof is identical to that in \cite{DJKZ}, we omit it here.

\begin{lemma} \label{Flecom} Fix $c \in \RR$. Then the following hold:  \\
(i) Let $\tilde y_1$ and $\tilde y_2$ be positive functions on $(0, s_*)$ satisfying
\be
\tilde y_1'(t) \leq \Gamma(t, \tilde y_1(t), c), \quad \tilde y_2'(t) \geq \Gamma(t, \tilde y_2(t), c) \quad a.e.~ t \in (0, s_*). 
\notag
\ee
If $\tilde y_1(0) \leq \tilde y_2(0)$, then $\tilde y_1(t) \leq \tilde y_2(t)$ for all $t \in [0, s_*]$. 

\smallskip

\noindent (ii) Let $\hat y_1$ and $\hat y_2$ be positive functions on $(s_*, 1)$ satisfying
\be
\hat y_1'(t) \leq \Gamma(t, \hat y_1(t), c), \quad \hat y_2'(t) \geq \Gamma(t, \hat y_2(t), c) \quad a.e. ~ t \in (s_*, 1). 
\notag
\ee
If $\hat y_1(1) \geq \hat y_2(1)$, then $\hat y_1(t) \geq \hat y_2(t)$ for all $t \in [s_*, 1]$. 
\end{lemma}
\begin{proof} See \cite[Lemma 4.5]{DJKZ}. 
\end{proof}

\smallskip

Lastly, we establish strict monotonicity with respect to $c$: for each fixed $t \in (0, s_*)$, the map $c \mapsto \tilde y_c(t): \RR \rightarrow \RR$ is strictly increasing, whereas for each fixed $t \in (s_*, 1)$, the map $c \mapsto \hat y_c(t): \RR \rightarrow \RR$ is strictly decreasing. The non-strict monotonicity follows directly from Lemma \ref{Flecom} (see \cite[Corollary 4.1]{DJKZ}). The strict form, however, is needed to extend the conclusion to the entire positivity intervals. We therefore employ an argument analogous to the uniqueness proof.

\begin{lemma} \label{mono}Let $c_1, c_2 \in \RR$ with $c_1 < c_2$.  Then $\tilde y_{c_1}(t) < \tilde y_{c_2}(t)$ for all $t \in (0, s_*)$, while $\hat y_{c_1}(t) > \hat y_{c_2}(t)$ for all $t \in (s_*,1)$.
\end{lemma}
\begin{proof}  We prove only the IVP. The TVP case is analogue.  Set
\be
\tilde z_{c_1}(t)=(\tilde y_{c_1}(t))^{1/p'}, \quad \tilde z_{c_2}(t)=( \tilde y_{c_2}(t))^{1/p'}, \quad \o(t):=\tilde z_{c_1}(t)- \tilde z_{c_2}(t), \quad t \in [0, s_*]. 
\notag
\ee  
Arguing as in the proof of Lemma \ref{Uniqueness2}, we obtain 
\be
\o^{\prime}(t)=c_1-c_2+\frac{f(t)}{(p-1)(\tilde z_{c_1}(t))^{\frac{1}{p-1}} (\tilde z_{c_2}(t))^{\frac{1}{p-1}}(\theta(t))^{1-\frac{1}{p-1}}} \o(t), \quad a.e. ~t \in (0, s_*), 
\notag
\ee
where $\theta(t)$ lies between $\tilde z_{c_1}(t)$ and $\tilde z_{c_2}(t)$. Define the integrating factor as in \eqref{IF}. Then we get 
\be
(\o F)^{\prime}(t)=(c_1-c_2)F(t) <0, \quad a.e. ~t \in (0, s_*). 
\notag
\ee
Thus $(\o F)(t)$ is strictly decreasing on $(0, s_*)$. Since $(\o F)(0)=0$, it follows that $\o(t)<0$ for all $t \in (0, s_*)$; that is, 
\be
\tilde z_{c_1}(t)<\tilde z_{c_2}(t) \quad \text{for all $t \in (0, s_*)$}.  
\notag
\ee
\end{proof}

\begin{remark}\label{mono_positive} The strict monotonicity with respect to $c$ extends to the entire positivity intervals of $\tilde y_c$ and $\hat y_c$. Let $c_1 < c_2$.  Assume,  for contradiction, that there exists $t_* \in (s_*, 1)$  such that  $0<\tilde y_{c_1}(t)  <  \tilde y_{c_2}(t)$ on $[s_*, t_*)$ and $\tilde y_{c_1}(t_*) = \tilde y_{c_2}(t_*)$. Viewing $\tilde y_{c_1}$ and  $\tilde y_{c_2}$ as solutions of the TVP on $[s_*, t_*]$ with terminal data at $t_*$,  the strict monotonicity in the backward setting implies $\tilde y_{c_1}(t) >  \tilde y_{c_2}(t)$ on $[s_*, t_*)$, contradicting the assumed inequality. 
\end{remark}

\smallskip

\section{Auxiliary  BVPs on $[0, s_*]$ and $[s_*, 1]$}\label{BVP on subintervals}

We now turn to the BVPs posed separately on $[0, s_*]$ and $[s_*, 1]$, depending on a parameter $c \in \RR$:
\begin{equation}\label{Fbvp}
\begin{cases}
y'(t) = p' \left[ (c-h(t)) (y (t))^{\frac{1}{p}} -f(t)  \right],  \quad t \in (0, s_*), \\
 y(0)=0=y(s_*), 
\end{cases}
\end{equation} 
and
\begin{equation}\label{Bbvp}
\begin{cases}
y'(t) = p' \left[ (c-h(t)) (y (t))^{\frac{1}{p}} -f(t)  \right],  \quad t \in (s_*,1), \\
 y(s_*)=0=y(1),
\end{cases}
\end{equation}
where $p>1$ with $\frac{1}{p}+\frac{1}{p'}=1$, $h \in L^{\infty}(0,1)$, and $f \in L^1(0,1)$ satisfying \eqref{fcondition} and \eqref{semicont}. Throughout this section, we additionally assume that $h$ satisfy \textbf{(H3)}. 

Our goal is to determine the values of $c$ for which the positive solution $\tilde y_c$ of the  IVP \eqref{y2} (resp. $\hat y_c$ of the TVP \eqref{y1}) also satisfies $\tilde y_c(s_*)=0$ (resp. $\hat y_c(s_*)=0$). Since solutions of \eqref{y2} and \eqref{y1} are positive on $(0,s_*)$ and $(s_*,1)$, respectively, the plus sign in $y^+$ has been dropped. 

The ultimate aim of this paper is to characterize values of $c$ that yield positive solutions of the full BVP \eqref{ode} on $[0,1]$. The auxiliary problems \eqref{Fbvp}–\eqref{Bbvp} serve to eliminate values of $c$ that cannot produce such solutions. Indeed, whenever $c$ admits solutions to \eqref{Fbvp} or \eqref{Bbvp}, they necessarily vanish at the interior point $s_*$, precluding any positive solution of \eqref{ode} on $[0,1]$.

In the monostable case \cite{DJKZ}, where $f>0$ on $(0,1)$, existence and nonexistence for the BVP have already been studied in detail. The results there apply directly to \eqref{Bbvp} on $[s_*,1]$. However, our later analysis revealed that the nonexistence range can be further extended (see Remark~\ref{Monononex} below). For this reason we slightly revise the nonexistence statement rather than repeating it verbatim. The problem \eqref{Fbvp} on $[0,s_*]$ is handled by symmetry. For completeness we present both existence and nonexistence results on $[0,s_*]$; the existence part follows directly from \cite{DJKZ}, while the nonexistence part is given here in a sharpened form.

Before proceeding to the detailed analysis, we first derive a necessary condition on the wave speed $c$ that must be satisfied by any solution of \eqref{Fbvp} and \eqref{Bbvp}, respectively.

\begin{lemma}\label{lFNC} 
If $\tilde y_c=\tilde y_c(t)$ is a positive solution of \eqref{Fbvp}, then 
\be \label{FNC}
c \leq h(s_*) \quad \text{and} \quad c<\frac{H(s_*)}{s_*}.
\ee
\end{lemma}
\begin{proof} Suppose, for contradiction, that $c>h(s_*)$. Since $h$ is continuous at $s_*$, there exists $\delta>0$ such that $c>h(t)$ for all $t \in [s_*-\delta, s_*]$. Integrating \eqref{Fbvp} over $[s_*-\delta, s_*]$ gives
\be
\tilde y_c(s_*-\delta)=- \int_{s_*-\delta}^{ s_*} \tilde y_c'(t) dt = -p'\Big[\int_{s_*-\delta}^{s_*} (c-h(t))(\tilde y_c(t))^{1/p}dt-\int_{s_*-\delta}^{s_*} f(t)dt \Big]<0,
\notag
\ee 
contradicting positivity of $\tilde y_c$. Hence $c\leq  h(s_*)$.  

Next, dividing \eqref{Fbvp} by $(\tilde y_c(t))^{1/p}$ and integrating over $[0, s_*]$ gives
\be
0=cs_*-\int_0^{s_*} h(t) dt - \int_0^{s_*} \frac{f(t)}{(\tilde y_c(t))^{1/p}}dt. 
\notag
\ee
Since $f<0$ on $(0, s_*)$, it follows that
$$c < \ds \frac{1}{s_*}\int_0^{s_*} h(t) dt=\frac{H(s_*)}{s_*}. $$
\end{proof}

In a similar manner, we obtain the following necessary condition for \eqref{Bbvp}. 

\begin{lemma}\label{lBNC} 
If $\hat y_c=\hat y_c(t)$ is a positive solution of \eqref{Bbvp}, then 
\be \label{BNC}
c \geq h(s_*) \quad \text{and} \quad c > \frac{H(1)-H(s_*)}{1-s_*}. 
\ee
\end{lemma}

\medskip

We are now ready to determine the values of $c$ such that the BVPs \eqref{Fbvp} and \eqref{Bbvp} admit positive solutions on $(0,s_*)$ and $(s_*,1)$, respectively. For convenience, we introduce the following notation, to be used in the subsequent analysis:
\be \label{Nmu}
\begin{split}
& \tilde \mu := \sup_{t \in (0, s_*)} \frac{-f(t)}{(s_*-t)^{p'-1}}, \quad  \hat \mu := \sup_{t\in (s_*,1)} \frac{f(t)}{(t-s_*)^{p'-1}}, \\
& \tilde \nu:=\liminf_{t \rightarrow s_*^-}\frac{-f(t)}{(s_*-t)^{p'-1}}, \quad \hat \nu:=\liminf_{t \rightarrow s_*^+} \frac{f(t)}{(t-s_*)^{p'-1}}.
\end{split}
\ee
We also recall that $h_m$ denotes the infimum of $h$ on $[0,s_*]$, while $h_M$ denotes the supremum of $h$ on $[s_*,1]$.

\begin{theorem} [Existence for \eqref{Fbvp}] \label{Fexthm}
Suppose that 
\be \label{Fmu}
0< \tilde \mu < + \infty.
\ee
Then, if $h(s_*)< \frac{H(s_*)}{s_*}$,  there exists a number
\be\label{c_F1}
 c_F \in [h_m -(p')^{\frac{1}{p'}} p^{\frac{1}{p} }\tilde \mu^{\frac{1}{p'}}, ~h(s_*)]
\ee
such that the problem \eqref{Fbvp} admits a unique positive solution if and only if $c \leq  c_F$. On the other hand, if $h(s_*) \geq \frac{H(s_*)}{s_*}$, the same result holds with 
\be\label{c_F2}
 c_F \in [h_m -(p')^{\frac{1}{p'}} p^{\frac{1}{p} }\tilde \mu^{\frac{1}{p'}}, ~\frac{H(s_*)}{s_*}). 
\ee
\end{theorem}
\begin{proof}  
We claim that $\tilde y_c(s_*) =0 $ for $c$ sufficiently small.  That is, for $c$ sufficiently small  a solution $\tilde y_c=\tilde y_c(t)$ on $[0, s_*]$ of \eqref{y2} turns out to be a positive solution of \eqref{Fbvp}.  For $c < h_m$, define
\be
\phi_c (s):= (h_m- c) (s_*-s)^{ \frac{1}{p} } -(s_*-s), \quad s \in (s_*-(h_m-c)^{p'}, s_*).
\notag
\ee
Then $\phi_c(s) >0$ for all $s \in (s_*-(h_m-c)^{p'}, s_*)$, and it satisfies $ \phi_c(s_*) = \phi_c(s_*-(h_m-c)^{p'})=0$. Moreover,  $\phi_c$ attains its maximum value
$$M_c := \left( \frac{h_m-c}{p} \right)^{p'}(p-1)~  \mbox{at the point} ~  k:= s_*-\left(\frac{h_m-c}{p} \right)^{p'} \in (s_*-(h_m- c)^{p'}, s_*).$$
A direct calculation shows that $h_m-c \geq (p')^{\frac{1}{p'}} p^{\frac{1}{p}} \tilde \mu^{\frac{1}{p'}}$ if and only if $ M_c \geq \tilde \mu$, which implies that for $c \leq h_m- (p')^{\frac{1}{p'}} p^{\frac{1}{p}} \tilde \mu^{\frac{1}{p'}}$, we have $\phi_c (k) \geq \tilde \mu$, that is,
\begin{equation} \label{Fbd}
(h_m-c) (s_*- k)^{\frac{1}{p}} -\tilde \mu \geq s_*- k.
\end{equation}
We define
\be
\varphi(t) := (s_*- k)(s_*-t)^{p'}, \quad  t \in [0, s_*].  
\notag
\ee
Then $\varphi(0) >0$, and it satisfies
\be
\begin{split}
& \varphi'(t)-\Gamma(t, \varphi(t), c) \\
& = -(s_*-k)p' (s_*-t)^{p'-1} -p' \Big[(c-h(t))(\varphi(t))^{\frac{1}{p}} -f(t)\Big]\\
& \geq \Big[ (c-h_m) (s_*-k)^{\frac{1}{p}} +\tilde \mu \Big]p' (s_*-t)^{p'-1} -p' \Big[(c-h_m)(\varphi(t))^{\frac{1}{p}} -f(t)\Big] \\
& \geq \Big[ (c-h_m) (s_*-k)^{\frac{1}{p}} + \tilde \mu \Big]p' (s_*-t)^{p'-1} -p' \Big[(c-h_m)(\varphi(t))^{\frac{1}{p}} + \tilde \mu (s_*-t)^{p'-1}  \Big] \\
& =0.
\end{split}
\notag
\ee
Here, we have used \eqref{Fmu} and \eqref{Fbd}. By Lemma \ref{Flecom}, it follows that
\be
0 \leq y_c(t ) \leq \varphi(t), \quad   t \in [0, s_*].
\notag
\ee
Since $\varphi(s_*)=0$, we conclude that $\tilde y_c(s_*)=0$. Therefore, for any $c \leq   h_m-(p')^{\frac{1}{p'}} p^{\frac{1}{p}} \tilde \mu^{\frac{1}{p'}}$, the unique solution $\tilde y_c=\tilde y_c(t)$ of \eqref{y2} is also a positive solution of \eqref{Fbvp}.

\smallskip

Now, recalling \eqref{FNC}, set
\be
c_F :=  \sup \{ c \leq h(s_*) \quad \text{and} \quad c<\frac{H(s_*)}{s_*} : ~ \mbox{\eqref{Fbvp} has a unique positive solution}\}.
\notag
\ee
From this definition, we immediately obtain the bounds
\be \label{rangecF}
h_m-(p')^{\frac{1}{p'}} p^{\frac{1}{p}} \tilde \mu^{\frac{1}{p'}} \leq c_F \leq \min\{h(s_*), \frac{H(s_*)}{s_*}\}. 
\ee
Let $\{ c_n\}$ be a sequence with $c_n \leq c_F$ such that $c_n \rightarrow c_F$,  and let $\tilde y_{c_n}=\tilde y_{c_n}(t)$ be a positive solution of \eqref{Fbvp}. By Lemma \ref{Cdop}, the sequence $\tilde y_{c_n}$ converges uniformly to $\tilde y_{c_F}$ on $[0, s_*]$, implying that  $\tilde y_{c_F}$ is also a positive solution of \eqref{Fbvp}. Consequently, the BVP \eqref{Fbvp} has a unique positive solution if and only if $c \leq c_F$. However, from \eqref{FNC},  we get that $c \neq \frac{H(s_*)}{s_*}$, which establishes \eqref{c_F2} in the case $h(s_*) \leq \frac{H(s_*)}{s_*}$.
\end{proof}

As noted at the beginning of this section, the range of $c$ for which no positive solution exists is larger than in the monostable case \cite{DJKZ}, although the analysis proceeds symmetrically. The next theorem makes this precise.

\begin{theorem} [Nonexistence for \eqref{Fbvp}] \label{Fnonexthm}
Suppose that
\be \label{Fnu}
 \tilde \nu >0.
\ee
If
\be \label{condition_nonexist}
h(s_*)-(p')^{\frac{1}{p'}} p^{\frac{1}{p} }\tilde \nu^{\frac{1}{p'}} < c < h(s_*),
\ee
then the problem \eqref{Fbvp} admits no positive solution. In particular, if $\tilde  \nu =+\infty$, then  \eqref{Fbvp} admits no positive solution for any $c < h(s_*)$.  
\end{theorem}
\begin{proof} We argue by contradiction. Fix any $c \in \mathbb{R}$ that satisfies \eqref{condition_nonexist}. Assume that the BVP \eqref{Fbvp} has a positive solution $\tilde y_c=\tilde y_c(t)$ in $(0, s_*)$.  Since $ -(p')^{\frac{1}{p'}}p^{\frac{1}{p}}\tilde \nu^{\frac{1}{p'}}< c -h(s_*)< 0$,  there exists a constant $L>0$ such that 
\be
-(p')^{\frac{1}{p'}}p^{\frac{1}{p}}\tilde \nu^{\frac{1}{p'}} < -L<c -h(s_*)<0. 
\notag
\ee
By the continuity of $h$ at $s_*$, there exists $t_0 \in [0,s_*)$ such that  
\be \label{range of c-h(t)}
-(p')^{\frac{1}{p'}}p^{\frac{1}{p}}\tilde \nu^{\frac{1}{p'}} < -L<c -h(t)<0 \quad \text{for all  $t \in [t_0, s_*]$}. 
\ee
Since $\tilde y_c$ is also a solution of \eqref{y2} on $[0, s_*]$, it follows from the uniqueness result that the restricted function $\tilde y_c:=\tilde y_c|_{[t_0, s_*]}$ is a unique solution of the following IVP
\be
\begin{cases}
 y'(t) = p' \left[ (c-h(t)) ( y^{+} (t))^{\frac{1}{p}} -f(t)  \right] \quad \text{a.e. in $(t_0, s_*)$},\\
 y(t_0)=\tilde y_c(t_0).
\end{cases}
\notag
\ee
We now define an operator $T: C[t_0, s_*] \rightarrow C[t_0, s_*]$ as
\be
T(u)(t):=p'\int_{s_*}^t\Big[(c-h(\tau))(u^+(\tau))^\frac{1}{p}-f(\tau)\Big] d\tau, \quad u \in C[t_0, s_*].
\notag
\ee
Then $\tilde y_c$ is definitely a fixed point of the operator $T$ in $C[t_0, s_*]$. Owing to the fact that $c-h(t) < 0$ on $(t_0, s_*]$,  we conclude that for any $u_1\leq u_2 $ on $[t_0, s_*]$,
\be
T(u_1)(t)-T(u_2)(t)=p'\int_{s_*}^{t}\Big[(c-h(\tau))\Big((u_1^+(\tau))^\frac{1}{p}-(u_2^+(\tau))^\frac{1}{p}\Big) \Big] d\tau \leq 0,
\notag
\ee
that is, $T$ is monotone increasing on $C[t_0, s_*]$. Set
\be
\bar y_0(t):=L^{p'}(s_*-t)^{p'}, \quad t \in [t_0, s_*]. 
\notag
\ee
Due to \eqref{range of c-h(t)}, we obtain that for all $t \in [t_0, s_*]$, 
\be
\begin{split}
T(\bar y_0)(t)
& = p'\int_{s_*}^{t}\Big[(c-h(\tau))  L^{\frac{p'}{p}}(s_*-\tau)^{\frac{p'}{p}} \Big] d\tau-p'\int_{s_*}^t f(\tau) d\tau \\
& \leq p'\int_{t}^{s_*}\Big[-(c-h(\tau))  L^{\frac{p'}{p}}(s_*-\tau)^{\frac{p'}{p}} \Big] d\tau \\
& \leq p'L^{1+\frac{p'}{p}}\int_{t}^{s_*}   (s_*-\tau)^{\frac{p'}{p}} d\tau   \\
& = \bar y_0(t). 
\end{split}
\notag
\ee
Then, by the monotonicity of $T$ on $C[t_0, s_*]$, a sequence defined as
\be \label{seq_y_n}
\bar y_{n+1}=T(\bar y_n), \quad n=0, 1, 2, \cdots,
\ee
satisfies that
\be
\bar y_0(t) \geq \bar y_1(t) \geq \cdots \geq \bar y_n(t) \geq \cdots,  \quad t \in (t_0, s_*),
\notag
\ee
and moreover it is bounded from below; namely
\be
\bar y_{n}=T(\bar y_{n-1}) \geq -p'\int_{s_*}^t f(\tau) d\tau, \quad n \in \mathbb{N}.
\notag
\ee
According to \cite[Theorem 6.3.16]{DM}, the sequence \eqref{seq_y_n} converges to the greatest fixed point of $T$. From the fact that $\tilde y_c$ is a fixed point of $T$, we deduce that
\be \label{mtildey}
\bar y_0(t) \geq \bar y_1(t) \geq \cdots \geq \bar y_n(t) \geq \cdots \geq \tilde y_c(t) >0, \quad t \in (t_0, s_*).
\ee

\smallskip

In order to obtain a contradiction we notice from \eqref{Fnu} and \eqref{range of c-h(t)} that there exists $\delta \in [t_0, s_*)$ and $\bar \nu \in \big(\frac{1}{p'p^{p'-1}}, 1\Big)$ such that
\be \label{ineq_f_nu}
-f(t) \geq \bar \nu L^{p'}(s_*-t)^{p'-1}, \quad t \in (\delta, s_*).
\ee
Recalling \eqref{seq_y_n} and using \eqref{ineq_f_nu}, a direct computation gives that
\be
\bar y_1(t) \leq L^{p'}(s_*-t)^{p'}(1-\bar \nu),  \quad t \in (\delta, s_*)
\notag
\ee
and
\be
\bar y_2(t) \leq L^{p'}(s_*-t)^{p'}\Big[(1-\bar \nu)^{\frac{1}{p}}-\bar \nu \Big], \quad t \in (\delta, s_*).
\notag
\ee
In general, for each $n \in \mathbb{N}$, we obtain
\be \label{mtildey2}
\bar y_n(t) \leq a_n L^{p'}(s_*-t)^{p'}, \quad  \quad t \in (\delta, s_*),
\ee
where $a_n$ is a sequence defined as
\be \label{seq_a_n}
a_0=1, \quad a_n=(a_{n-1})^{\frac{1}{p}}-\bar \nu.
\ee
It follows from \eqref{mtildey} and \eqref{mtildey2} that the sequence $\{ a_n \}$ is decreasing and bounded below by $\tilde y_c$, and hence it converges to $a_{\infty} \in (0,1)$ that satisfies $\bar \nu=a_{\infty}^{\frac{1}{p}}-a_{\infty}=a_{\infty}^{\frac{1}{p}}(1-a_{\infty})^{\frac{1}{p'}}$. A simple calculation shows that
\be
\bar \nu \leq \max_{x \in (0,1)} x^{\frac{1}{p}}(1-x)^{\frac{1}{p'}}=\frac{1}{p'p^{p'-1}},
\notag
\ee
a contradiction to the fact that $\bar \nu \in \big(\frac{1}{p'p^{p'-1}}, 1\big)$. Therefore, \eqref{Fbvp} cannot have a positive solution. In particular, if $\tilde{\nu} = + \infty$, \eqref{condition_nonexist} yields that \eqref{Fbvp}  has no positive solution for any $c < h(s_*)$.  
\end{proof}

\medskip

We now state the existence and nonexistence results for the BVP \eqref{Bbvp} on $[s_*, 1]$, which are fully symmetric to Theorem \ref{Fexthm} and Theorem \ref{Fnonexthm}.

\begin{theorem} [Existence for \eqref{Bbvp}] \label{exthm}
Suppose that
\be \label{mu}
0< \hat \mu  < + \infty.
\ee
Then, if $h(s_*)> \frac{H(1)-H(s_*)}{1-s_*}$, then there exists a number
\be\label{c_B1}
c_B \in [ h(s_*), ~ h_M +(p')^{\frac{1}{p'}} p^{\frac{1}{p} }\hat \mu^{\frac{1}{p'}}]
\ee
such that the problem \eqref{Bbvp} admits a unique positive solution if and only if $c\geq c_B$. On the other hand, if $h(s_*)\leq \frac{H(1)-H(s_*)}{1-s_*}$, then the same result holds with 
\be \label{c_B2}
c_B \in  ( \frac{H(1)-H(s_*)}{1-s_*} , ~ h_M +(p')^{\frac{1}{p'}} p^{\frac{1}{p} }\hat \mu^{\frac{1}{p'}} ].
\ee 
\end{theorem}

\smallskip

\begin{theorem}[Nonexistence for \eqref{Bbvp}] \label{Nonexistence} 
Suppose that
\be \label{nu}
 \hat \nu >0.
\ee
If
\be \label{condition_nonexist2}
h(s_*) < c < h(s_*)+(p')^{\frac{1}{p'}}p^{\frac{1}{p}}\hat \nu^{\frac{1}{p'}},
\ee
then the problem \eqref{Bbvp} admits no positive solution. In particular, if $\hat \nu =+\infty$, then \eqref{Bbvp} admits no positive solution for any $c > h(s_*)$.
\end{theorem}

\begin{remark} \label{Monononex} We notice that even in the monostable case treated in \cite[Theorem 4.2]{DJKZ}, the nonexistence range of $c$ can actually be strengthened to
\be
h(0)<c<h(0)+(p')^{\frac{1}{p'}}p^{\frac{1}{p}}\nu^{\frac{1}{p'}}. 
\notag
\ee 
Furthermore, in all subsequent remarks and theorems following \cite[Theorem 4.2]{DJKZ}, the term $h_m+(p')^{\frac{1}{p'}}p^{\frac{1}{p}}\nu^{\frac{1}{p'}}$ can be replaced by $h(0)+(p')^{\frac{1}{p'}}p^{\frac{1}{p}}\nu^{\frac{1}{p'}}$. In particular, the second plot in \cite[Fig.2]{DJKZ} should therefore be disregarded (a revised version of \cite{DJKZ} is available on arXiv:2601.12869). 
\end{remark}

\smallskip

\begin{remark} \label{remarkbvp}
According to the preceding four theorems, under the assumptions \eqref{Fmu} and \eqref{mu}, both $\tilde \nu$ and $\hat \nu$ are also finite, and thus threshold wave speeds $c_B$ and $c_F$ are defined as finite values. Moreover, they satisfy the estimates:  
\be \label{rcF}
h_m-(p')^{\frac{1}{p'}} p^{\frac{1}{p}} \tilde \mu^{\frac{1}{p'}} \leq c_F \leq \min\{h(s_*)-(p')^{\frac{1}{p'}} p^{\frac{1}{p}} \tilde \nu^{\frac{1}{p'}} , \frac{H(s_*)}{s_*}\} \leq  h(s_*)-(p')^{\frac{1}{p'}} p^{\frac{1}{p}} \tilde \nu^{\frac{1}{p'}} ,
\ee
and 
\be \label{rcB}
h(s_*)+(p')^{\frac{1}{p'}} p^{\frac{1}{p}} \hat \nu^{\frac{1}{p'}}  \leq \max\{h(s_*)+(p')^{\frac{1}{p'}} p^{\frac{1}{p}} \hat \nu^{\frac{1}{p'}} , \frac{H(1)-H(s_*)}{1-s_*} \}  \leq  c_B \leq h_M+(p')^{\frac{1}{p'}} p^{\frac{1}{p}}\hat \mu^{\frac{1}{p'}}. 
\ee
Thus, whenever \eqref{Fmu} and \eqref{mu} hold, we have 
\be
-\infty<c_F \leq c_B <\infty.  
\notag
\ee
Furthermore, if $c\geq c_B$ or $c\leq c_F$, then the BVP \eqref{ode} has no positive solution, since in these cases necessarily $\tilde y_c(s_*)=0$ or $\hat y_c(s_*)=0$. Consequently, a necessary condition for \eqref{ode} to possess a positive solution is  
\be \label{neode}
c_F<c_B,
\ee
which implies that any admissible wave speed $c$ must lie strictly between $c_F$ and $c_B$. 
\end{remark}

\begin{remark}\label{extended cFcB}
For clarity in what follows, we regard $c_F$ and $c_B$ as elements of the extended real line.
Without the assumptions \eqref{Fmu} and \eqref{mu}, the finiteness of $c_F$ and $c_B$ is not guaranteed; in particular, one may have $c_F = -\infty$ or $c_B = +\infty$.
\end{remark}

\smallskip

\section{Existence and nonexistence for the first-order BVP on $[0,1]$} \label{results of the first order ode}

Building on the results of Sections \ref{Forward and Backward Initial Value Problems} and \ref{BVP on subintervals}, we now investigate the existence and nonexistence of positive solutions to the BVP \eqref{ode} on the entire interval $[0,1]$. 

Before proceeding with the existence analysis, we first establish that the strict ordering \eqref{neode} is indeed the necessary and sufficient condition for the existence of a unique positive solution to \eqref{ode}. That is, a unique positive solution to the BVP \eqref{ode} exists only when the wave speeds admitting positive solutions to the BVP \eqref{Fbvp} on $[0, s_*]$ are strictly smaller than those admitting positive solutions to the BVP \eqref{Bbvp} on $[s_*, 1]$. The next proposition makes this approach more  precise; 
its proof follows the strategy of \cite[Theorem~12]{MMM04}, adapted to our setting. For this purpose we introduce the following notation. Define the strip
\be
\mathcal S:=\{(t, v) : 0 \leq t \leq 1, ~ v \geq 0\}. 
\notag
\ee
For each $c \in (c_F, c_B)$, let $(0, \tilde s_c) \subseteq (0,1)$ denote the maximal interval on which the forward trajectory $t \mapsto (t, \tilde y_c(t))$ remains in $\mathcal S$, and let $(\hat s_c, 1) \subseteq (0,1)$ denote the maximal interval on which the backward trajectory $t \mapsto (t, \hat y_c(t))$ remains in $\mathcal S$. By Lemma \ref{Fpositiveness}, 
\be
s_*< \tilde s_c \leq 1, \quad 0 \leq \hat s_c <s_*. 
\notag
\ee
We then set 
\be
\mathcal A_0:=\{ c \in (c_F, c_B) : \tilde s_c <1 \}, \quad \mathcal A_1:=\{ c \in (c_F, c_B) : \hat s_c >0 \}. 
\notag
\ee
That is, $c \in \mathcal A_0$ means the forward trajectory leaves $\mathcal S$ before reaching $t=1$, whereas $c \in \mathcal A_1$ means the backward trajectory leaves $\mathcal S$ before reaching $t=0$.

We collect basic properties of $\mathcal A_0$ and $\mathcal A_1$:
\begin{itemize}
\item The forward and backward trajectories vanish at the respective exit points:
for every $c \in \mathcal A_0$, $\tilde y_c(\tilde s_c)=0$;
for every $c \in \mathcal A_1$, $\hat y_c(\hat s_c)=0$.

\smallskip

\item  The trajectories of $\tilde y_c$ and $\hat y_c$ do not intersect each other in the interior of the strip $\mathcal S$, i.e.,
\be
\mathcal A_0 \cap \mathcal A_1 = \emptyset.
\notag
\ee
This follows from the uniqueness of positive solutions on their respective intervals (see Remark~\ref{uniqueness of ps}).

\smallskip

\item There exists at most one value $c^* \in (c_F, c_B)$ such that $c^* \notin \mathcal A_0 \cup \mathcal A_1$.
Indeed, if $c \notin \mathcal A_0 \cup \mathcal A_1$, then $\tilde s_c = 1$ and $\hat s_c = 0$, and hence $\tilde y_c \equiv \hat y_c$ on $[0,1]$ by uniqueness.
By the monotonicity of $\tilde y_c$ and $\hat y_c$ in $c$ (Remark~\ref{mono_positive}), such coincidence can occur for at most one value of $c$.

\smallskip

\item Elements in $\mathcal A_0$ lie to the left of those in $\mathcal A_1$; equivalently,
if $c_0 \in \mathcal A_0$ and $c_1 \in \mathcal A_1$, then $c_0 < c_1$. This is obvious from the monotonicity of $\tilde y_c$ and $\hat y_c$ on $c$ (Remark~\ref{mono_positive}).

\smallskip

\item By the uniqueness observed in Remark~\ref{uniqueness of ps}, whether $c$ belongs to $\mathcal A_0$ or $\mathcal A_1$ determines the endpoint behavior of the trajectories.
If $c \in \mathcal A_0$, then the backward solution satisfies $\hat s_c = 0$ and $\hat y_c(0) > 0$;
if $c \in \mathcal A_1$, then the forward solution satisfies $\tilde s_c = 1$ and $\tilde y_c(1) > 0$.
Conversely, these endpoint conditions also characterize the sets $\mathcal A_0$ and $\mathcal A_1$:
if $\hat s_c = 0$ with $\hat y_c(0) > 0$ for some $c$, then $c \in \mathcal A_0$;
if $\tilde s_c = 1$ with $\tilde y_c(1) > 0$ for some $c$, then $c \in \mathcal A_1$.
Hence, we may equivalently write
\be
\mathcal A_0 = \{ c \in (c_F, c_B) : \hat s_c = 0,~ \hat y_c(0) > 0 \}, \qquad
\mathcal A_1 = \{ c \in (c_F, c_B) : \tilde s_c = 1,~ \tilde y_c(1) > 0 \}.
\notag
\ee

\end{itemize}

With these observations, we now establish the following proposition, which constitutes a key step toward the proof of our main results. 

\begin{proposition} \label{existenceprop} Let $f$ and $h$ satisfy \eqref{pf}, \eqref{semicont}, and \textbf{(H3)}. Then the BVP \eqref{ode} admits a unique positive solution $y=y(t)$ on $(0,1)$ if and only if $c_F < c_B$. 
\end{proposition}
\begin{proof} If $c \leq c_F$ or $c \geq c_B$, then by \eqref{rcF}--\eqref{rcB} we have $\tilde y_c(s_*)=0$ or $\hat y_c(s_*)=0$, and hence \eqref{ode} admits no positive solution on $(0,1)$. Conversely, assume $c_F < c_B$. We first prove that both $\mathcal A_0$ and $\mathcal A_1$ are nonempty, considering separately the four cases determined  by the finiteness of $c_F$ and $c_B$ (see Remark~\ref{extended cFcB}).

\smallskip

\noindent Case I: $-\infty < c_F < c_B < \infty$. Assume, by contradiction, $\mathcal A_0 = \emptyset$. Then, by the monotonicity of $\tilde y_c$ in $c$, we must have $\mathcal A_1=(c_F, c_B)$; that is, for any $c \in (c_F, c_B)$, $\tilde s_c=1$ and $\tilde y_c(1)>0$. Choose a sequence $\{c_n\} \subset \mathcal A_1$ with $c_n \downarrow c_F$. Then by continuous dependence (Lemma \ref{Cdop}) together with the monotonicity of  $\tilde y_{c}$ in $c$,  
\be \label{convergence}
\tilde y_{c_n}(t) \downarrow \tilde y_{c_F}(t) \quad \text{on $[0, s_*]$}.
\ee
Define a function
\be
\tilde y(t)=
\begin{cases}
\tilde y_{c_F}(t), \quad & t \in [0, s_*), \\
\hat y_{c_1}(t), \quad & t \in [s_*, 1]. 
\end{cases}
\notag
\ee 
Since $\tilde y_{c_n}(t) > \tilde y(t)$ on $(0, 1]$ for all $n \in \mathbb{N}$,  the convergence \eqref{convergence} yields that $$0=\tilde y_{c_F}(s_*) \geq \tilde y(s_*)>0,$$ a contradiction. In a similar manner,  $\mathcal A_1 \neq \emptyset$.

\smallskip

\noindent Case II: $-\infty=c_F < c_B< \infty$. By the same argument as in Case I, $\mathcal A_1 \neq \emptyset$. To prove  that $\mathcal A_0 \neq \emptyset$, it suffices to show that $\tilde s_c <1$ for sufficiently small $c$. Fix $c<h_m$. Integrating \eqref{y2} from $0$ to $\tilde s_c$ gives
\be
\begin{split}
0 \leq \tilde y_c(\tilde s_c) 
&= p'\int_0^{\tilde s_c} (c-h(t)) (\tilde y_c (t))^{\frac{1}{p}}  dt - p'\int_0^{\tilde s_c}  f(t) dt \\
& \leq p' (c-h_m) \int_0^{\tilde s_c} (\tilde y_c (t))^{\frac{1}{p}} dt - p'\int_0^{\tilde s_c}  f(t) dt \\
& <  - p'\int_0^{\tilde s_c}  f(t) dt, 
\end{split}
\notag
\ee
so that $\ds\int_0^{\tilde s_c}  f(t) dt<0$. Thanks to \eqref{pf}, we get $\tilde s_c <1$, and hence $c \in \mathcal A_0$ for all $c<h_m$. 

\smallskip

\noindent Case III: $-\infty<c_F < c_B= \infty$. The proof that $\mathcal A_0 \neq \emptyset$ is the same as in Case I. Assume, for contradiction, that $\mathcal A_1 = \emptyset$. Then, since  $A_0$ is connected (by the monotonicity of $\tilde y_c$ in $c$), we have $\mathcal A_0=(c_F, \infty)$. Fix $c_0>c_F$ and set 
\be
M:=\max\Big \{c_0, ~h_M+\frac{\int_0^1 f(t) dt}{\int_0^{s_*} (\tilde y_{c_0}(t))^{\frac{1}{p}}dt} \Big \}.
\notag 
\ee  
For any $c> c_0$, we get $\tilde y_c(t) >\tilde y_{c_0}(t)$ on $[0, s_*]$. Moreover, since $f>0$ on $(s_*, 1)$ and $s_*<\tilde s_c<1$ for all $c \in (c_F, \infty)=\mathcal A_0$,  integrating \eqref{y2} from $0$ to $\tilde s_c$ yields that for all $c>M$, 
\be
\begin{split}
0 
& = \tilde y_c(\tilde s_c)-\tilde y_c(0) \\ 
&= p'\int_0^{\tilde s_c} (c-h(t)) (\tilde y_c (t))^{\frac{1}{p}}  dt - p'\int_0^{\tilde s_c}  f(t) dt \\
& \geq  p' (c-h_M) \int_0^{s_*} (\tilde y_c (t))^{\frac{1}{p}} dt - p'\int_0^{1}  f(t) dt \\
& \geq  p' (c-h_M) \int_0^{s_*} (\tilde y_{c_0} (t))^{\frac{1}{p}} dt - p'\int_0^{1}  f(t) dt \\
&>0,
\end{split}
\notag
\ee
a contradiction. Thus, $\mathcal A_1 \neq \emptyset$.

\smallskip 

\noindent Case IV: $-\infty=c_F<c_B= \infty$. Observe that the proof of $\mathcal A_0 \neq \emptyset$ in Case~II does not rely on $\mathcal A_1 \neq \emptyset$, whereas the proof of $\mathcal A_1 \neq \emptyset$ in Case~III does use $\mathcal A_0 \neq \emptyset$. Therefore, we first obtain $\mathcal A_0 \neq \emptyset$ by the same argument as in Case~II, and then deduce $\mathcal A_1 \neq \emptyset$ by fixing $c_0 \in \mathcal A_0$, exactly as in Case~III.

\smallskip

Consequently, in all four cases we have established that both $\mathcal A_0$ and $\mathcal A_1$ are nonempty. We now define 
\be
c^*:=\sup \mathcal A_0=\inf \mathcal A_1
\notag
\ee
and show that $c^* \notin   \mathcal A_0 \cup \mathcal A_1$. To argue as in Case I, choose a sequence $\{c_n\} \subset \mathcal A_1$ with $c_n \downarrow c^*$. Then, for any $n \in \mathbb{N}$, we have $\tilde s_{c_n}=1$ and $\tilde y_{c_n}(1)>0$, while $\tilde y_{c_n}(t)$ is bounded below on $[0,1]$ by
\be
\tilde y(t)=
\begin{cases}
\tilde y_{c^*}(t), \quad & t \in [0, s_*), \\
\hat y_{c_1}(t), \quad & t \in [s_*, 1]. 
\end{cases}
\notag
\ee
Since $\hat y_{c_1}(1)=0$, we obtain that $\tilde y_{c_n}(t) \downarrow \tilde y_{c^*}(t)$ uniformly on $[0, 1]$. Hence $\tilde s_{c^*}=1$ and $\tilde y_{c^*}(1) \geq 0$, so $c^* \notin \mathcal A_0$. Similarly, we take $\{d_n\} \subset \mathcal A_0$ with $d_n \uparrow c^*$. Since for every $n \in \mathbb{N}$, $\hat y_{d_n}(t)$ is bounded below on $[0, 1]$ by 
\be
\hat y(t)=
\begin{cases}
\tilde y_{d_1}(t), \quad & t \in [0, s_*], \\
\hat y_{c^*}(t), \quad & t \in (s_*, 1], 
\end{cases}
\notag
\ee
the same reasoning shows $c^* \notin \mathcal A_1$.
\end{proof}

\bigskip

We are now in a position to introduce two existence results and one nonexistence result for the BVP \eqref{ode}.
All sufficient conditions for existence and nonexistence in the case $p=2$ were studied in \cite{MMM04}, but we show that these conditions remain valid under our weak regularity assumptions, too.
The conditions \eqref{Excondition_1}, \eqref{Excondition_2}, and \eqref{existencecondition} stated in Theorem \ref{existencebvp} and Theorem \ref{existencebvp2} guarantee the strict ordering $c_F < c_B$, which in turn ensures the existence of a positive solution to the BVP \eqref{ode}, thanks to Proposition \ref{existenceprop}.

\begin{theorem}[Existence for \eqref{ode}] \label{existencebvp} Let $f$ and $h$ satisfy \eqref{pf}, \eqref{semicont}, and \textbf{(H3)}.  If either condition
\be\label{Excondition_1}
\liminf_{t \rightarrow s_*^-}\frac{(s_*-t)^{p'-1}(h(s_*)-h(t))}{-f(t)}>-\infty
\ee
or
\be\label{Excondition_2}
\limsup_{t \rightarrow s_*^+}\frac{(t-s_*)^{p'-1}(h(s_*)-h(t))}{f(t)}<\infty 
\ee
holds, then the problem \eqref{ode} admits a unique positive solution $y=y(t)$ on $(0,1)$ if and only if $c=c^*$, where $c^*$ satisfies 
\be \label{range}
h_m-(p')^{\frac{1}{p'}} p^{\frac{1}{p}} \tilde \mu^{\frac{1}{p'}} < c^*  < h_M+(p')^{\frac{1}{p'}} p^{\frac{1}{p}} \hat \mu^{\frac{1}{p'}}
\ee  
with $h_m$, $h_M$, $\tilde \mu$, and $\hat \mu$ defined in \eqref{maxminh} and \eqref{Nmu}, respectively.
\end{theorem}
\begin{proof} By Proposition \ref{existenceprop}, it suffices to prove that $c_F < c_B$ in the case where both $c_F$ and $c_B$ are finite. In view of Remark \ref{remarkbvp}, this reduces to showing that either $c_F<h(s_*)$ or $h(s_*)<c_B$. 

Suppose that \eqref{Excondition_2} holds. Then there exist constants $\delta>0$ and $K>0$ such that 
\be \label{Excondition_2_r}
h(s_*)-h(t) < K \frac{f(t)}{(t-s_*)^{p'-1}} \quad \text{for all $t \in (s_*, s_*+\delta)$}. 
\ee
To argue by contradiction, assume that $h(s_*)=c_B$. Then,  as noted in Remark \ref{remarkbvp}, we have  
\be \label{hatnu0}
 0=\ds \hat \nu=\liminf_{t \rightarrow s_*^+}\frac{f(t)}{(t-s_*)^{p'-1}}. 
\ee
By setting $\hat z_c(t)=(\hat y_c(t))^{1/p'}$, the differential equation \eqref{y1} can be written as
\be
\frac{d}{dt} (\hat z_{h(s_*)}(t))^{p'}=(h(s_*)-h(t))(\hat z_{h(s_*)}(t))^{p'-1}-f(t). 
\notag
\ee
Integrating this both sides over $(s_*, s_*+\delta)$ and using the condition$\hat z_{h(s_*)}(s_*)=0$, we obtain from \eqref{Excondition_2_r} that 
\be \label{60}
\begin{split}
0<\big(\hat z_{h(s_*)}(s_*+\delta)\big)^{p'}
& = \int_{s_*}^{s_*+\delta} (h(s_*)-h(t))(\hat z_{h(s_*)}(t))^{p'-1}dt-\int_{s_*}^{s_*+\delta} f(t) dt  \\
&  \leq K \int_{s_*}^{s_*+\delta} \frac{f(t)}{(t-s_*)^{p'-1}}(\hat z_{h(s_*)}(t))^{p'-1} dt-\int_{s_*}^{s_*+\delta} f(t)dt \\
& = \int_{s_*}^{s_*+\delta} f(t) \Big[K\Big(\frac{\hat z_{h(s_*)}(t)}{t-s_*} \Big)^{p'-1}-1 \Big]dt. 
\end{split}
\ee
To derive a contradiction, we now claim that 
\be \label{limit0}
\lim_{t \rightarrow s_*^+} \frac{\hat z_{h(s_*)}(t)}{t-s_*}=0.  
\ee
If this claim holds, then by choosing $\delta>0$ sufficiently small,  the right--hand side of \eqref{60} becomes strictly negative, yielding the desired contradiction. 

Thanks to \eqref{hatnu0}, we first note that 
\be
\limsup_{t \rightarrow s_*^+}~ p' \Big(\frac{\hat z_{h(s_*)}(t)}{t-s_*} \Big)^{p'-1} \big(\hat z_{h(s_*)}'(t)-(h(s_*)-h(t))\big)= -\liminf_{t \rightarrow s_*^+}\frac{f(t)}{(t-s_*)^{p'-1}}=0. 
\notag
\ee
By setting
\be
g(t):=\Big(\frac{\hat z_{h(s_*)}(t)}{t-s_*} \Big)^{p'-1} \big(\hat z_{h(s_*)}'(t)-(h(s_*)-h(t))\big), 
\notag
\ee
we obtain that 
\be \label{derivativeH}
\begin{split}
\frac{d}{dt} (\hat z_{h(s_*)}(t))^{p'}
& = p'(t-s_*)^{p'-1}g(t)+p'(\hat z_{h(s_*)}(t))^{p'-1}(h(s_*)-h(t)) \quad \text{a.e. in $(s_*, 1)$}.
\end{split}
\ee
Since $\limsup_{t \rightarrow s_*^+} g(t)=0$ and $h$ is continuous at $s_*$, for any $\varepsilon>0$ there exists $\gamma>0$ such that
\be
g(t) \leq \varepsilon \quad \text{and} \quad |h(t)-
h(s_*)| \leq \varepsilon \quad \text{a.e. in $(s_*, s_*+\gamma)$}. 
\notag
\ee
Substituting this into \eqref{derivativeH} yields 
\be
\frac{d}{dt} (\hat z_{h(s_*)}(t))^{p'} \leq p'\varepsilon(t-s_*)^{p'-1}+p'\varepsilon (\hat z_{h(s_*)}(t))^{p'-1} \quad \text{a.e. in $(s_*, s_*+\gamma)$}. 
\notag
\ee
Since $\hat z_{h(s_*)}(s_*)=0$, integrating over $(s_*, t)$ gives that for all $t \in (s_*, s_*+\gamma)$
\be \label{34}
(\hat z_{h(s_*)}(t))^{p'} \leq \varepsilon (t-s_*)^{p'}+p'\varepsilon \int_{s_*}^t (\hat z_{h(s_*)}(\tau))^{p'-1} d\tau. 
\ee
Set 
\be
u(t):=\sup_{s_*<\tau<t} \hat z_{h(s_*)}(\tau). 
\notag
\ee
By  continuity of $\hat z_{h(s_*)}$, $u(t)<\infty$ and $u(t) \rightarrow 0$ as $t \rightarrow s_*^+$. Moreover, 
\be
\int_{s_*}^t (\hat z_{h(s_*)}(\tau))^{p'-1} d\tau \leq (t-s_*)(u(t))^{p'-1}.
\notag 
\ee
Applying this to \eqref{34} yields that for all $t \in (s_*, s_*+\gamma)$, 
\be
(\hat z_{h(s_*)}(t))^{p'} \leq \varepsilon (t-s_*)^{p'}+p'\varepsilon (t-s_*)(u(t))^{p'-1}, 
\notag
\ee
so that 
\be
(u(t))^{p'} \leq \varepsilon (t-s_*)^{p'}+p'\varepsilon (t-s_*)(u(t))^{p'-1}. 
\notag
\ee
Dividing by $(t-s_*)^{p'}$, we obtain 
\be \label{70}
\Big(\frac{u(t)}{t-s_*}\Big)^{p'-1}\Big(\frac{u(t)}{t-s_*}- p'\varepsilon \Big) \leq  \varepsilon. 
\ee
Hence, we conclude that for all $t \in (s_*, s_*+\gamma)$, 
\be \label{80}
\frac{u(t)}{t-s_*} \leq p'\varepsilon+\varepsilon^{\frac{1}{p'}}. 
\ee 
Indeed, if $\ds \frac{u(t)}{t-s_*} > p'\varepsilon+\varepsilon^{\frac{1}{p'}}$, then
\be
\ds \frac{u(t)}{t-s_*} -p'\varepsilon > \varepsilon^{\frac{1}{p'}} \quad \text{and} \quad \Big(\frac{u(t)}{t-s_*}\Big)^{p'-1}>\varepsilon^{1-\frac{1}{p'}}, 
\notag
\ee
so that the left--hand side of \eqref{70} would exceed $\varepsilon$, a contradiction. 
Taking $\limsup$ in \eqref{80} and letting $\varepsilon \rightarrow 0^+$ gives
\be
\limsup_{t \rightarrow s_*^+} \frac{u(t)}{t-s_*} =0. 
\notag
\ee
Since $0<\hat z_{h(s_*)}(t)  \leq u(t)$ for all  $t \in (s_*, s_*+\gamma)$, we conclude that  
\be
0 \leq \limsup_{t \rightarrow s_*^+} \frac{\hat z_{h(s_*)}(t)}{t-s_*} \leq \limsup_{t \rightarrow s_*^+} \frac{u(t)}{t-s_*} =0. 
\notag
\ee
This completes the proof of the claim \eqref{limit0}. From \eqref{rcF}--\eqref{rcB}, the inequality \eqref{range} is evident because $c^* \in (c_F, c_B)$. 
\end{proof}

\bigskip

In addition to the existence result established under the conditions \eqref{Excondition_1} and \eqref{Excondition_2}, we can also obtain another existence result under the alternative condition \eqref{existencecondition}. This condition represents a weaker form of the monotonicity (nondecreasing) assumption on $h$ near $s_*$, which guarantees that the conditions \eqref{Excondition_1} and \eqref{Excondition_2} are satisfied.

\begin{theorem}[Existence for \eqref{ode}] \label{existencebvp2} Let $f$ and $h$ satisfy \eqref{pf}, \eqref{semicont}, and \textbf{(H3)}.  If $h$ further satisfies that 
\be \label{existencecondition}
\int_{s_*}^{t_*} (h(s_*)-h(t))dt \leq 0 \quad  \text{for some $t_* \in [0,1]\setminus\{s_*\}$}, 
\ee
then the problem  \eqref{ode} admits a unique positive solution $y=y(t)$ if and only if $c=c^*$, where $c^*$ satisfies \eqref{range}.
\end{theorem}
\begin{proof} It suffices to prove that $c_F < h(s_*)$ or $h(s_*) < c_B$, ensuring that $c_F<c_B$.  Suppose that \eqref{existencecondition} holds for some $t_* \in [0, s_*)$. Let $\tilde y_{h(s_*)}(t)$ be a positive solution of the forward initial value problem \eqref{y2} on $(0, s_*)$ for $c=h(s_*)$, and define 
\be
\tilde z_{h(s_*)}(t):=(\tilde y_{h(s_*)}(t))^{1/p'}>0 \quad \text{on $(0, s_*)$}. 
\notag
\ee 
Then we have
\be
\tilde z_{h(s_*)}'(t) = h(s_*)-h(t)-\frac{f(t)}{(\tilde z_{h(s_*)}(t))^{1/(p-1)}} \quad \text{a.e. in $(0, s_*)$}. 
\notag
\ee
Integrating both sides over $(t_*, s_*)$, and using \eqref{existencecondition} together with $f<0$ on $(0, s_*)$, we obtain 
\be
\tilde z_{h(s_*)}(s_*)-\tilde z_{h(s_*)}(t_*) = \int_{t_*}^{s_*} (h(s_*)-h(t))dt-\int_{t_*}^{s_*} \frac{f(t)}{(\tilde z_{h(s_*)}(t))^{1/(p-1)}}dt >0. 
\notag
\ee
Since $\tilde z_{h(s_*)}(t_*)\geq 0$, we have $\tilde z_{h(s_*)}(s_*)>0$ from which we deduce that $\tilde y_{h(s_*)}(s_*)>0$. Hence, $c_F < h(s_*)$.  In a similar manner, if the assumption \eqref{existencecondition} holds for some $t_* \in (s_*, 1]$, then we conclude that $h(s_*)<c_B$.
\end{proof}

As the final theorem, we provide the conditions \eqref{nonh}–\eqref{nonFf} that guarantee $c_F = c_B$, which in turn ensures that the BVP \eqref{ode} admits no positive solution for any $c \in \RR$. 

\begin{theorem}[Nonexistence for \eqref{ode}]\label{nonexistencebvp} Let $f$ and $h$ satisfy \eqref{pf}, \eqref{semicont}, and \textbf{(H3)}.  Further suppose that 
\be \label{nonh}
\ds \int_{s_*}^t (h(s_*)-h(\tau))d\tau>0 \quad  \text{for all $t \in [0,1]\setminus\{s_*\}$}
\ee
holds together with 
\be \label{nonBf}
f(t)<\frac{1}{p'p^{p'/p}}(h(s_*)-h(t)) \Big[\int_{s_*}^t (h(s_*)-h(\tau))d\tau\Big]^{\frac{p'}{p}}, \quad t \in (s_*, 1),
\ee
and 
\be \label{nonFf}
f(t)>\frac{1}{p'p^{p'/p}}(h(s_*)-h(t)) \Big[ \int_{s_*}^t (h(s_*)-h(\tau))d\tau \Big]^{\frac{p'}{p}}, \quad t \in (0,s_*).
\ee
Then the problem \eqref{ode} admits no positive solution for any $c \in \RR$. 
\end{theorem}
\begin{proof} We first observe that the given assumptions \eqref{nonh}-\eqref{nonFf} imply
\be
\frac{H(1)-H(s_*)}{1-s_*} < h(s_*) <  \frac{H(s_*)}{s_*}, 
\notag
\ee
since $h(s_*)>h(t)$ for all $t \in (s_*, 1)$ and $h(s_*)<h(t)$  for all $t \in (0, s_*)$. 

\smallskip

To complete the proof, it suffices to show that $\tilde y_{h(s_*)}=0$ and $\hat y_{h(s_*)}=0$, thanks to the monotonic dependence of $\tilde y_c$ and $\hat y_c$ on $c$.  For this purpose, we define a function 
\be \label{K}
K(t)=\Big[ \frac{1}{p}\int_{s_*}^t (h(s_*)-h(\tau)) d\tau \Big]^{p'}, \quad \quad  t \in [0,1].
\ee
Then, $K(t) > 0$ for  all $t \in [0,1]\setminus \{s_0\}$, and we compute
\be
\begin{split}
K'(t)
& =\frac{p'}{p} (h(s_*)-h(t)) \Big[\frac{1}{p}\int_{s_*}^t (h(s_*)-h(\tau)) d\tau \Big]^{\frac{p'}{p}} \\
& = p'(h(s_*)-h(t))\Big[\frac{1}{p}\int_{s_*}^t (h(s_*)-h(\tau)) d\tau \Big]^{\frac{p'}{p}}- (h(s_*)-h(t)) \Big[\frac{1}{p}\int_{s_*}^t (h(s_*)-h(\tau)) d\tau \Big]^{\frac{p'}{p}} \\
& = p'(h(s_*)-h(t))\Big[\frac{1}{p}\int_{s_*}^t (h(s_*)-h(\tau)) d\tau \Big]^{\frac{p'}{p}}-\frac{p'}{p'p^{p'/p}} (h(s_*)-h(t)) \Big[\int_{s_*}^t (h(s_*)-h(\tau)) d\tau \Big]^{\frac{p'}{p}}. 
\notag
\end{split}
\ee

\smallskip

We first consider the case $c \leq h(s_*)$. From \eqref{nonFf}, we have
\be \label{K1}
K'(t) > p'[(h(s_*)-h(t))K(t)^{1/p}-f(t)]. 
\ee
Let $\tilde y_{h(s_*)}$ denote the positive solution of the IVP \eqref{y2} on $(0, s_*)$ corresponding to $c = h(s_*)$. Since $K(0) > 0 = \tilde y_{h(s_*)}(0)$, Lemma~\ref{Flecom}(i) together with \eqref{K1} yields
\be
0 \leq \tilde y_{h(s_*)}(t) \leq K(t),  \quad t \in [0, s_*]. 
\notag
\ee
Because $K(s_*) = 0$, it follows that $\tilde y_{h(s_*)}(s_*) = 0$. By Lemma~\ref{mono}, for any $c \leq h(s_*)$, the positive solution $\tilde y_c$ of \eqref{y2} satisfies $\tilde y_c(s_*)=0$. Hence, there is no positive solution to the BVP \eqref{ode} for any $c \leq h(s_*)$.

\smallskip

Similarly, consider the case $c \geq h(s_*)$. Let $\hat y_{h(s_*)}$ be the positive solution of the TVP \eqref{y1} on $(s_*, 1)$ for $c = h(s_*)$. Using \eqref{nonh}, \eqref{nonBf}, and Lemma~\ref{Flecom}(ii), we get $\hat y_{h(s_*)}(s_*)=0$. Applying Lemma \ref{mono}, we conclude that $\hat y_c(s_*)=0$ for any $c \geq h(s_*)$.  This completes the proof. 

\end{proof}

\medskip

Theorems \ref{existencetws}–\ref{nonexistencetws} are immediate consequences of the previous results. Indeed, under assumptions \textbf{(H1)}–\textbf{(H3)}, they are obtained from Proposition~\ref{prop} and Theorems~\ref{existencebvp}–\ref{nonexistencebvp} by setting $f(t) := (d(t))^{\frac{1}{p-1}} g(t)$. The properties \textup{(i)}–\textup{(iv)} in Theorem \ref{existencetws} follow directly from Remark \ref{property of U} and together with the proof of Proposition \ref{prop} (see \cite[Section 3.2]{DJKZ}).  

\bigskip

\section*{Appendix}

In this appendix, we prove that for $1<p \leq 2$, both one-sided derivatives $U'(z^-)$ and $U'(z^+)$ are nonzero when $U(z)=s_*$, provided that the profile is monotone. Hence, the nonvanishing assumptions in Definition \ref{sdef}(a) and (b) can be removed. By \eqref{Ldef}, if one of $U'(z^-)$ and $U'(z^+)$ is zero, then so is the other. Therefore, it suffices to prove that $U'(z) \neq 0$ whenever $U(z)=s_*$.

\begin{theorem}\label{thmapp}  Let $1<p \leq 2$, and let $U$ be a monotone profile of \eqref{odeU} for some $c \in \RR$. Then $U'(z)<0$ whenever $U(z)=s_*$.
\end{theorem}
\begin{proof} Suppose that there exists $z^* \in (z_0, z_1)$ such that $U(z^*)=s_*$ and $U'(z^*)=0$. If there is an open interval $(a, b) \subset (z_0, z_1)$ on which $U(z)=s_*$, we redefine $z^*=b$. Then the restriction $U|_{(z^*, z_1)}$ has a continuous strictly decreasing inverse 
\be
z=U^{-1}: (0, s_*) \longrightarrow (z^*, z_1).
\notag
\ee
For $U \in (0, s_*)$, set
\be \label{app2}
t:=U,  \quad w(t):=v(z(U)), \quad \text{and} \quad y(t):=|w(t)|^{p'}.
\ee
Then $y=y(t)$ is a positive solution of the first-order BVP
\begin{equation}\label{odeyappendix}
\begin{cases}
y'(t) = p' \left[ (c-h(t)) (y(t))^{\frac{1}{p}} -(d(t))^{\frac{1}{p-1}}g(t)  \right],  \quad \text{a.e. $t \in (0,s_*)$}, \\
 y(0)=0=y(s_*).
\end{cases}
\end{equation}
Here, the boundary condition $y(s_*)=0$ follows from $U'(z^*)=0$ and \eqref{v}; see \cite[Section 3.2]{DJKZ} for details of this reduction. 

Since $(d(t))^{\frac{1}{p-1}}g(t) <0$ on $(0, s_*)$ and $h$ is bounded on $(0, 1)$, the ODE in \eqref{odeyappendix} yields
\be
\frac{d}{dt} (y(t))^{\frac{1}{p'}} > c-h(t) > c-\sup_{0<t<s_*}h(t) \geq -\gamma, \quad \text{a.e. $t \in (0,s_*)$}
\notag
\ee
for some $\gamma>0$. For a given $t \in (0, s_*)$, integrating over $(t, s_*)$ gives
\be \label{appen3}
 (y(t))^{\frac{1}{p'}} \leq \gamma(s_*-t) \quad \text{for all $t \in (0,s_*)$}.
\ee
Since $(z^*, z_1) \cap M_U$ consists of only finitely many points,
choose $U_0 \in (0,s_*)$ such that $z|_{(U_0,  s_*)} \in C^1(U_0, s_*)$. Then, by \eqref{v}, \eqref{app2} and \eqref{appen3}, 
\be
\begin{split}
b-z(U_0) 
& =\int_{U_0}^{s_*} z'(U) dU=\int_{U_0}^{s_*} -\Big|\frac{d(U)}{w(U)}\Big|^{\frac{1}{p-1}} dU \\
& =\int_{U_0}^{s_*} - \frac{(d(U))^{\frac{1}{p-1}}}{(y(U))^{\frac{1}{p}}}dU \leq -\frac{1}{\gamma^{\frac{1}{p-1}}}\int_{U_0}^{s_*} \frac{(d(U))^{\frac{1}{p-1}}}{(s_*-U)^{\frac{1}{p-1}}}dU.
\end{split}
\notag
\ee
If $1<p \leq 2$, the right-hand side equals $-\infty$, which is a contradiction. Hence $U'(z^*) \neq 0$.  
\end{proof}

\bigskip

\begin{remark} For $p>2$, under the additional assumption that
\be \label{appen4}
-(d(t))^{\frac{1}{p-1}}g(t) \geq K(s_*-t)^{\beta} \quad \text{for all $t \in (s_*-\delta, s_*)$}
\ee
for some $K>0$, $\beta \in (0, \frac{1}{p-1})$, and some  $\delta>0$, one can also prove the statement of Theorem \ref{thmapp}. Indeed, assume $U'(z^*)=0$ and use the same notations as in Theorem \ref{thmapp}. Then it follows from \eqref{appen3} that 
\be
y'(t) >-p'\gamma^{\frac{1}{p-1}+1}(s_*-t)^{\frac{1}{p-1}}+p'(d(t))^{\frac{1}{p-1}}g(t) \quad \text{a.e. $t \in (s_*-\delta,s_*)$}.
\notag
\ee
By \eqref{appen4}, this yields
\be
y'(t) >-p'\gamma^{\frac{1}{p-1}+1}(s_*-t)^{\frac{1}{p-1}}+p'K(s_*-t)^{\beta} \quad \text{a.e. $t \in (s_*-\delta,s_*)$}.
\notag
\ee
Since $\beta<\frac{1}{p-1}$, by choosing $\delta=\delta(\gamma)>0$ smaller if necessary, we may ensure that 
\be
K(s_*-t)^{\beta}> \gamma^{\frac{1}{p-1}+1}(s_*-t)^{\frac{1}{p-1}} \quad \text{for all $t \in (s_*-\delta, s_*)$},
\notag
\ee
and hence 
\be
y'(t)>0 \quad \text{a.e. $t \in (s_*-\delta,s_*)$}.
\notag
\ee
This contradicts the fact that $y(s_*)=0$, and therefore the nonvanishing follows.  

Consequently, under condition \eqref{appen4}, the statement of Theorem \ref{thmapp} can be proved for all $1<p<\infty$, and in this case the nonvanishing assumptions in Definition \ref{sdef}(a) and (b)  can be completely removed. 

\end{remark}

\end{document}